\documentclass{rmmcart}
\usepackage{graphicx,overpic}
\usepackage[left=3cm, top=2cm, right=3cm, bottom=3cm]{geometry}
\usepackage{listings}
\usepackage{float}
\usepackage{amssymb,amsmath,mathtools}
\usepackage{enumerate}
\usepackage[numbers,square]{natbib}
\usepackage{wrapfig}
\usepackage{braket}
\usepackage{tikz-cd} 

\usepackage[utf8]{inputenc}
\usepackage[english]{babel} 

\usepackage{amsthm}
\newtheorem{thm}{Theorem}[section]
\newtheorem{lem}[thm]{Lemma}
\newtheorem{prop}[thm]{Proposition}
\newtheorem{cor}[thm]{Corollary}
\numberwithin{equation}{section}

\theoremstyle{definition}
\newtheorem{defn}[thm]{Definition} 
\newtheorem{remk}[thm]{Remark}

\newenvironment{claim}[1]{\par\noindent\underline{Claim:}\space#1}{}
\newenvironment{claimproof}[1]{\par\noindent\underline{Proof:}\space#1}{\leavevmode\unskip\penalty9999 \hbox{}\nobreak\hfill\quad\hbox{$\blacksquare$}}

\makeatletter
\newenvironment{breakproof}[1][\proofname]{\par
  \pushQED{\qed}%
  \normalfont \topsep6\p@\@plus6\p@\relax
  \trivlist
  \item[\hskip\labelsep
        \itshape
    #1\@addpunct{.}]\ignorespaces\item
}{%
  \popQED\endtrivlist\@endpefalse
}
\makeatother

\makeatletter
\newsavebox{\@brx}
\newcommand{\llangle}[1][]{\savebox{\@brx}{\(\m@th{#1\langle}\)}%
	\mathopen{\copy\@brx\kern-0.5\wd\@brx\usebox{\@brx}}}
\newcommand{\rrangle}[1][]{\savebox{\@brx}{\(\m@th{#1\rangle}\)}%
	\mathclose{\copy\@brx\kern-0.5\wd\@brx\usebox{\@brx}}}
\makeatother

\title{Agol's theorem on hyperbolic cubulations}

\author{Sam Shepherd}
\address{Mathematical Institute, University of Oxford, Oxford, United Kingdom}
\email{sam.shepherd@balliol.ox.ac.uk}

\date{October, 29, 2020}


\begin{document}
	
\begin{abstract}
	Agol proved that hyperbolic cubulated groups are virtually special. The aim of these notes is to make the proof accessible to a wider audience; we retain the underlying ideas and constructions of Agol, but substantially change or add to many parts of the argument to give a more transparent and detailed account.
\end{abstract}
\maketitle

\section{Introduction}

More precisely, the following theorem was proven by Agol in \cite{VHC}. Throughout we will always write `Agol's paper' to refer to \cite{VHC}.

\begin{thm}\label{mainthm}(Agol's Theorem)\\
Let $G$ be a hyperbolic group acting properly and cocompactly on a CAT(0) cube complex $X$. Then $G$ has a finite index subgroup $G'$ that acts freely on $X$ such that the quotient $X/G'$ is special.
\end{thm}

\subsection{Background}
	
Agol's Theorem spans two areas of activity within geometric group theory. On the one hand there has been a lot of work about \textit{cubulating} groups: finding for a given group a proper cocompact action on a CAT(0) cube complex. On the other hand there is the theory of special cube complexes.

The definition of CAT(0) cube complex goes back to Gromov \cite{Gromov}, but the topic of cubulating groups, which has been driven largely by Dani Wise, begins with Sageev's construction of a CAT(0) cube complex from a pocset, which provides a cubulation of a group $G$ if it acts on a pocset in a certain way (see \cite{Sageev} for instance). Many classes of groups have been cubulated using this method, such as the Coxeter groups in \cite{coxeter} and small cancellation groups in \cite{smallcanc}, and we will see more examples later in the introduction. A general criterion for cubulating hyperbolic groups was given in \cite{boundarycrit}. Consequences of cubulating a group include that the group is bi-automatic \cite{biautomatic}, and that it satisfies the Tits Alternative, namely every subgroup is either virtually finitely generated abelian or has a non-abelian free subgroup \cite{titsalt}.

Special cube complexes were introduced by Haglund and Wise in \cite{HW}. We say that a compact nonpositively curved cube complex is \emph{virtually special} if it has a finite-sheeted special cover, and we say that a group is \emph{virtually special} if it has a finite index subgroup that is the fundamental group of a compact special cube complex. Virtually special groups have even stronger properties than cubulated groups; to start with it was proven in \cite{HW} that they are virtually subgroups of right-angled Artin groups, they are subgroups of $SL_n(\mathbb{Z})$, and they are residually finite. 
Cube complexes contain certain characteristic subspaces called \emph{walls} (see Section \ref{sec:background}), and a principal way of analysing a cubulated group $G$ is to study its wall stabilisers (which form a finite collection of subgroups up to conjugation). If $G$ is hyperbolic and cubulated, then it is virtually special if and only if every wall stabiliser is separable \cite{HW}; moreover, this is also equivalent to every quasi-convex subgroup of $G$ being separable. A strengthened Tits Alternative was proven for virtually special groups in \cite[14.10]{qconvexh}, namely every subgroup is virtually abelian or large (has a finite index subgroup surjecting to $\mathbb{Z}\ast\mathbb{Z}$).

Previously, many hyperbolic groups were known to be cubulated but not known to be virtually special, so Agol's Theorem confirmed that all of these groups are in fact virtually special, and hence satisfy all the strong properties discussed above. Examples of such groups include metric small cancellation groups \cite{smallcanc} and closed hyperbolic 3-manifold groups - the latter can be cubulated using Kahn and Markovich's construction of quasi-fuchsian subgroups, together with work of Bergeron and Wise \cite{boundarycrit}. Small cancellation groups arise as random groups with respect to the Gromov density model, so Agol's Theorem implies that at low density random groups are virtually special, and thus residually finite. Combined with the cubulations of Niblo and Reeves \cite{coxeter}, Agol's Theorem also recovers the result of Haglund and Wise that hyperbolic Coxeter groups are virtually special. Other important classes of hyperbolic groups have been cubulated since Agol's Theorem, so again it could be deduced straight away that they are virtually special. For example Hagen and Wise cubulated hyperbolic free-by-cyclic groups in \cite{free-by-cyclic}, and Manning, Mj and Sageev cubulated certain hyperbolic surface-by-free groups in \cite{surface-by-free}.

The example of closed hyperbolic 3-manifold groups mentioned above is particularly significant. For $M$ a closed hyperbolic 3-manifold, the virtual specialness (and consequent largeness) of $\pi_1(M)$ established by Agol's Theorem provided an affirmative solution to Thurston's long-standing conjecture that $M$ has a finite-sheeted cover that fibres over the circle (see \cite{icm} for more details on this). It also provided the final piece in the puzzle to prove the Virtual Haken Conjecture, that $M$ has a finite-sheeted Haken cover, hence the title to Agol's paper. A more subtle consequence of Agol's Theorem is that $\pi_1(M)$ is good in the sense of Serre \cite{Serre}; this has significant implications for the profinite rigidity of 3-manifold groups - see \cite{BF} and \cite{BR}.

A major open problem in geometric group theory is the question of whether all hyperbolic groups are residually finite. Agol's Theorem gives a step towards a positive answer by showing that all hyperbolic cubulated groups are residually finite. Note that finite extensions of hyperbolic cubulated groups (i.e. groups that surject to a hyperbolic cubulated group with finite kernel) are still hyperbolic and cubulated, so residual finiteness is closed under finite extensions for these groups - something which is unknown for hyperbolic groups in general. Agol's Theorem can never show that all hyperbolic groups are residually finite however, as there are examples of hyperbolic groups with Kazhdan's Property (T), which have a global fixed point for any action on a CAT(0) cube complex, as follows from \cite{T}. Such groups include uniform lattices in quaternionic hyperbolic spaces and random groups in Gromov's density model for certain medium densities.

In \cite{qconvexh} Wise defines the \emph{quasi-convex hierarchy} of hyperbolic groups, which starts with finite groups and builds up via graphs of groups with quasi-convex edge groups (see Definition \ref{QVH}). He proves that a hyperbolic group is in this hierarchy if and only if it is cubulated and virtually special. This result has been used extensively in the literature to show that various hyperbolic groups are cubulated and virtually special, and it plays a key role in the proof of Agol's Theorem. Moreover, one can view Agol's Theorem as saying that this hierarchy is as large as possible, i.e. that it contains all hyperbolic cubulated groups.

Agol's theorem has been generalised to relatively hyperbolic groups by Or\'egon-Reyes \cite{relhyp1} and Groves and Manning \cite{relhyp2}, each with slightly different (but equivalent) assumptions on how the peripheral subgroups act on the cube complex $X$. These generalisations imply the virtual specialness of non-compact finite-volume hyperbolic 3-manifold groups and of limit groups, recovering theorems of Wise \cite{qconvexh}. In \cite{specialcellstab} Groves and Manning generalise Agol's Theorem in a different way by proving that any hyperbolic group acting cocompactly on a CAT(0) cube complex with quasi-convex and virtually special cell stabilisers is virtually special. A more general form of residual finiteness, introduced by Agol in \cite{Agol} and that is satisfied by virtually special groups, is being residually finite rationally solvable (RFRS). In \cite{Kielak} Kielak proved a group theoretic analogue of Thurston's Virtual Fibering Conjecture (which was settled by Agol's Theorem), showing that virtually RFRS groups with vanishing first $L^2$-Betti number are virtually fibred in the sense that they have finite index subgroups that surject to $\mathbb{Z}$.

Another active direction of current research in the field is studying non-hyperbolic virtually special groups. For instance, it is known that hyperbolic Coxeter groups are virtually special, but it is unknown whether all Coxeter groups are virtually special, or even virtually cubulated. More generally, there is as yet no analogue to Wise's quasi-convex hierarchy for non-hyperbolic groups, in particular it is not well understood when an amalgam of two virtually special groups is virtually special - Huang and Wise have a result in this direction \cite{specialgraph}, but it requires that the vertex groups are hyperbolic.
There is another sort of hierarchy that virtually special groups do admit: in \cite{HHG} Behrstock, Hagen and Sisto define the notion of a hierarchically hyperbolic group, and prove that these include all virtually special groups. Hierarchical hyperbolicity is a way of unifying techniques and concepts across quite a broad class of groups, including many non-cubulated groups like mapping class groups.

Agol poses many open questions in his ICM article \cite{icm}, as does Wise in his ICM article \cite{openprobs}. For example, given a compact nonpositively curved cube complex $X$, is there an algorithm that decides whether $X$ is virtually special? And is virtual specialness independent of the choice of cubulation? i.e. if $X$ is a compact nonpositively curved cube complex and $\pi_1 X$ is virtually special, is $X$ virtually special? It would also be interesting to have more examples of compact nonpositively curved cube complexes that are not virtually special - the currently known examples can all be traced to irreducible complete square complexes such as in \cite{wise}. Agol's Theorem implies that compact hyperbolic 3-manifold groups are $\mathbb{Z}$-linear, but it is unknown whether this extends to all compact 3-manifold groups (Liu \cite{Liu} and Przytycki-Wise \cite{PW} show that a compact aspherical  3-manifold has virtually special fundamental group if and only if it is nonpositively curved, so this goes beyond a question of virtual specialness).
And finally, do all closed hyperbolic manifolds of dimension greater than three have cubulated fundamental groups? (If so they would be virtually special.)

\subsection{Proof strategy}

The strategy of the proof of Theorem \ref{mainthm} is to cut $X$ along all of its walls and start with orbit representatives of these pieces; then glue these back together one wall at a time until we get a cube complex with universal cover $X$ and deck transformations a subgroup of $G$. At each stage we will have a collection of finite virtually special cube complexes whose universal covers are convex subspaces of $X$. To ensure the complexes remain virtually special at each stage we employ theorems from \cite{qconvexh} and \cite{comb}. The main condition we need for these theorems is that the walls we want to glue along are acylindrical subspaces; this requires keeping preimages of these gluing walls far apart in the universal cover. This is where the key idea of Agol's proof comes in: we endow each vertex in each complex with local data that colours nearby walls in the universal cover in such a way that: 
\begin{itemize}
\item finitely many colours are used,
\item the colours determine which walls are preimages of gluing walls,
\item walls that are close have different colours,
\item adjacent vertices have compatible colouring data, in that they agree about the colours of nearby walls,
\item if there are two vertices with compatible colouring data next to different gluing walls, then all vertices next to these walls have compatible colouring data (allowing us to `zip' the walls together).
\end{itemize}
That last point is important as it ensures that the colouring data retains its properties from one stage of the construction to the next. Devising colourings that satisfy all this is difficult, in fact it requires a more complicated definition of local colouring than simply colouring in walls that are near to the vertex in the universal cover - and the presence of infinite walls in the universal cover further messes things up. So what we actually do is pass to a quotient complex $\mathcal{X}=X/K$ with finite walls, and we colour walls in $\mathcal{X}$ rather than $X$ using local colourings that form a sort of `hierarchy'. The existence of a suitable $\mathcal{X}$ relies on knowing that the wall stabilisers for the action of $G$ on $X$ already satisfy Theorem \ref{mainthm} (which we are safe to assume by induction on the dimension of $X$) - plus it depends on the appendix to Agol's paper, which in turn relies on Wise's Malnormal Special Quotient Theorem. Note that we do not revisit the proofs contained in the appendix to Agol's paper. The hierarchy of coloured walls we use depends on Wise's quasi-convex hierarchy of hyperbolic groups that we discussed earlier.

\subsection{Summary of sections}
	
\begin{enumerate}
\setcounter{enumi}{1}
\item Background on walls: We provide some basic theory about walls in cube complexes, with a focus on CAT(0) cube complexes.
\item Special cube complexes: We recall some results about special cube complexes, which underpin many later arguments.
\item Making walls finite: This section constructs the quotient complex $\mathcal{X}$ with finite walls.
\item Invariant colouring measures: The main result here is the existence of a measure on the space of colourings of a graph.
\item Colouring walls: This is where we define the local colouring data.
\item Starting the gluing construction: Here we set up the main inductive construction, which at each stage gives a collection of finite virtually special cube complexes. We use Section 5 to carefully choose colouring data for the first stage of the construction, to ensure that later on we can always pair up gluing walls with compatible colouring data.
\item Controlling boundary walls: Here we prove the zipping property of gluing walls mentioned above, and prove that gluing walls are acylindrical subspaces.
\item Gluing up walls: In this final section we perform the inductive step by gluing the walls together; the subtlety here is that two gluing walls with compatible colouring data might be different shapes (so the zipping gets stuck), to solve this we take finite covers of our complexes.
\end{enumerate}
The sections with the most significant deviations from Agol's proof are 4, 7, 8 and 9.

\textbf{Acknowledgements:}\,
Thanks to Ric Wade and my supervisor Martin Bridson for their careful proofreading and helpful comments. Thanks also to the referee for further comments and suggestions. And thanks to Daniel Woodhouse for explaining Theorem \ref{thm:command} to me, which I used to simplify some of the arguments in Section \ref{sec:gluing}.

\section{Background on walls}\label{sec:background}

In this section we cover some basic definitions and results about cube complexes and their walls; most of this material can be found in \cite{HW}. 

 \begin{defn}(Cube complex and metric)\\
 A \textit{cube complex} is a metric polyhedral complex in which all polyhedra are unit Euclidean cubes. For $X$ a cube complex, $V(X)$ will denote the vertex set and $E(X)$ will denote the edge set. All edges in these notes will be unoriented. We will denote the metric by $d$ - and also use $d$ for the distance between subsets of a cube complex, $d(A,B)=\inf\{d(x,y)|\,x\in A,\, y\in B\}$. Any finite dimensional cube complex is a complete geodesic space \cite[I.7.33]{nonpos}. 
 
 Another metric which is commonly used on the vertex set of a cube complex, referred to as the $\ell_1$ or \emph{combinatorial metric}, defines the distance between two vertices to be the length of a shortest edge path between them. We will not use this metric directly, although we will often consider shortest edge paths.
 \end{defn}

\begin{defn} (Walls)\\
An $n$-cube $C=[-1,1]^n\subset \mathbb{R}^n$ has $n$ midcubes, each obtained by setting one coordinate to zero. The face of a midcube of $C$ is naturally the midcube of a face of $C$, and so the collection of midcubes of a cube complex $X$ can be given the structure of a cube complex. We call this the \textit{wall complex}; it has a natural immersion $q: \mathcal{W}\to X$ induced by inclusions of the midcubes - note this immersion is neither combinatorial nor an embedding, but it can be made combinatorial by passing to the barycentric subdivisions of $\mathcal{W}$ and $X$. A component of the wall complex is called a \textit{wall} (it is also called a hyperplane by other authors), but when talking about a wall we will usually be referring to its image under $q$. A picture of a wall in a cube complex is given below.
\end{defn}
\begin{wrapfigure}{r}{.5\textwidth}
	\centering
	\includegraphics[width=0.4\textwidth, clip=true ]{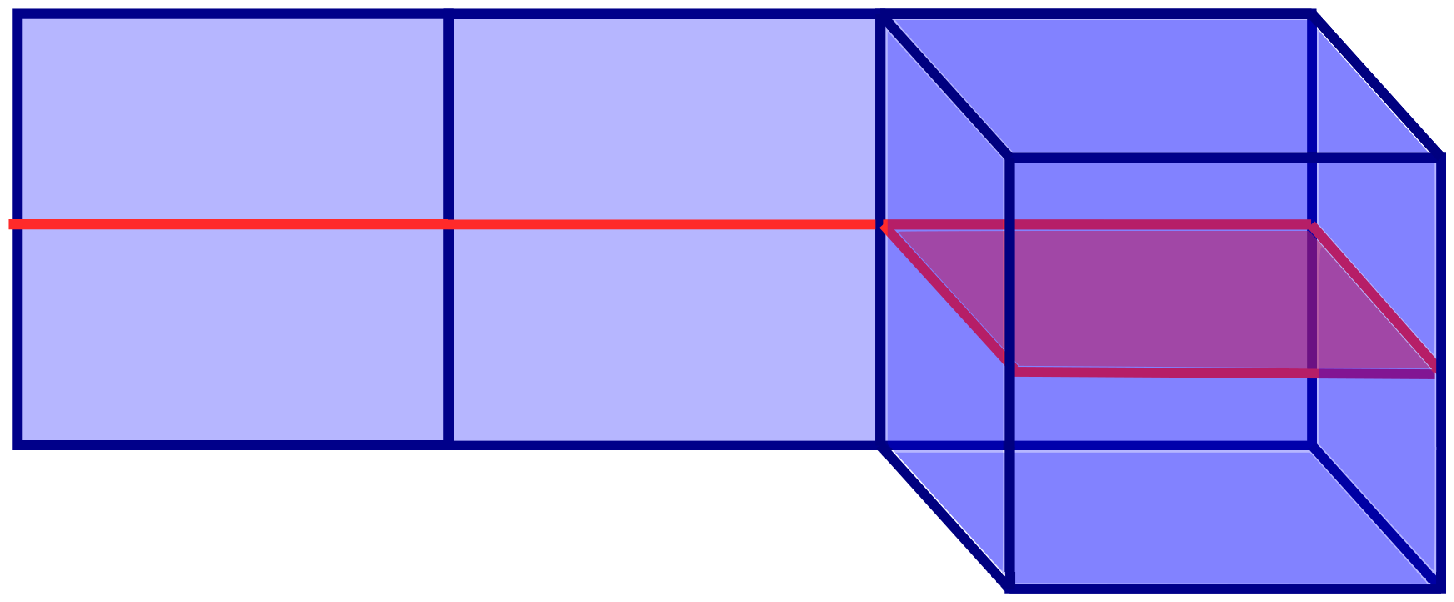}
\end{wrapfigure}

For an edge $e$ in a cube complex $X$, write $W(e)$ for the unique wall that it intersects, and say that $e$ is \textit{dual} to $W(e)$. If a group $G$ acts on $X$ then it also acts on the set of walls, and clearly $g W(e)=W(g e)$ for $g\in G$, $e\in E(X)$. For a wall $W$, let $G_W$ denote its stabiliser. Given an edge $e$ dual to $W$, $g\in G_W$ if and only if $W(ge)=W$ - thus $G_W$ is also the stabiliser of the set of edges dual to $W$.

\begin{remk}\label{cocompactstabiliser}
If the action of $G$ on $X$ is proper and cocompact then so is the action of $G_W$ on $W$. Properness is immediate, and cocompactness is a consequence of the following easy argument. Let $\{e_1,...,e_k\}$ be edges dual to $W$ which are representatives of those $G$-orbits of edges in $X$ that include edges dual to $W$. For an edge $e$ dual to $W$, there exists $g\in G$ with $ge=e_j$ some $1\leq j\leq k$, but then $g\in G_W$. Therefore the union of midcubes that intersect $\{e_1,...,e_k\}$ is a compact subset of $W$ with $G_W$-translates that cover $W$.
\end{remk}

\begin{defn}(NPC and CAT(0) cube complexes)\\
A cube complex is \textit{nonpositively curved} (we will use the shorthand NPC) if the link of each vertex is a flag complex, and an NPC cube complex is \textit{CAT(0)} if it is simply connected. For the general definition of a CAT(0) space, and why it is equivalent to ours for cube complexes, see \cite[II.1 and II.5.20]{nonpos}.
\end{defn}

For any cube complex $X$, the map $q:\mathcal{W}\to X$ is a local isometry. If $X$ is CAT(0) we have that, for any wall $W$, $q:W\to X$ is an embedding with convex image (because in a CAT(0) space, local geodesics  are geodesics and geodesics are unique), and so by identifying $W$ with its image we can view $W$ as a closed convex subspace of $X$. In particular, each cube of $X$ will have at most one midcube belonging to $W$.

\begin{defn}\label{twoside}
	For $W$ a wall in any cube complex, let $N(W)$ be the union of cubes that intersect $W$. Define an equivalence relation on the vertices of $N(W)$ in which $x\sim y$ if $x, y$ are joined by an edge path in $N(W)$ that never crosses $W$.   
\end{defn}

\begin{prop}\label{1or2sides}
	If $W$ is an \textit{embedded} wall in $X$, meaning that $q:W\to X$ is an embedding, then there will either be one or two $\sim$-equivalence classes. If there are two then we say that $W$ is \textit{two-sided} and we denote the two classes of vertices by $W^+$ and $W^-$.
\end{prop}
\begin{proof}\renewcommand{\qedsymbol}{}
\end{proof}
	\begin{wrapfigure}{l}{.5\textwidth}
		\centering
		\includegraphics[width=0.4\textwidth, clip=true ]{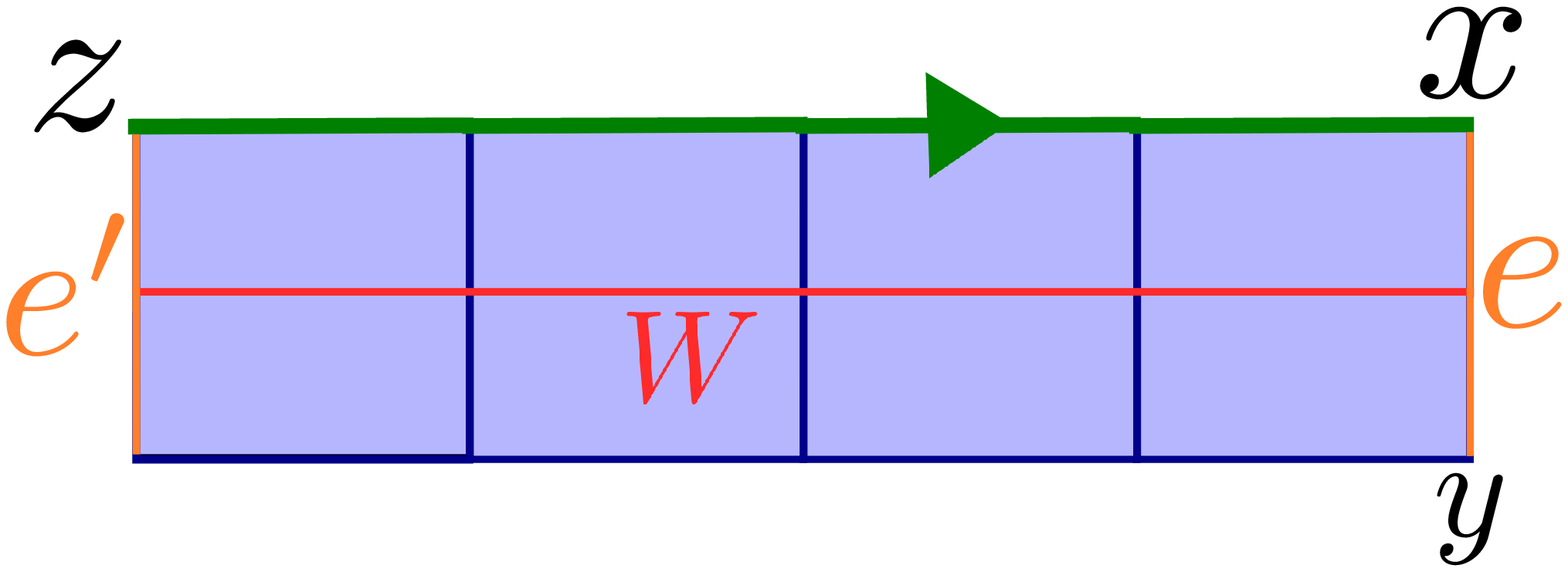}
	\end{wrapfigure}
Let $e$ be an edge with endpoints $x,y$ that is dual to $W$. Given a vertex $z\in N(W)$, we wish to show that it is equivalent to either $x$ or $y$. Take an edge $e'$ incident at $z$ that is dual to $W$, and an edge path in the cube structure of $W$ joining $e'$ with $e$. As shown, this edge path sits inside a sequence of squares containing $x,y,z$, and along the boundary of these squares there is an edge path in $X$ from $z$ to one of $x,y$ - and this path doesn't cross $W$ because each square of $X$ contains at most one midcube of $W$.\qed

\begin{prop}\label{twosides}
In a CAT(0) cube complex $X$, each wall is two-sided and separates $X$ into two connected components.
\end{prop}
\begin{proof}\renewcommand{\qedsymbol}{}
\end{proof}
\begin{wrapfigure}{r}{.5\textwidth}
\centering
\includegraphics[width=0.4\textwidth, clip=true ]{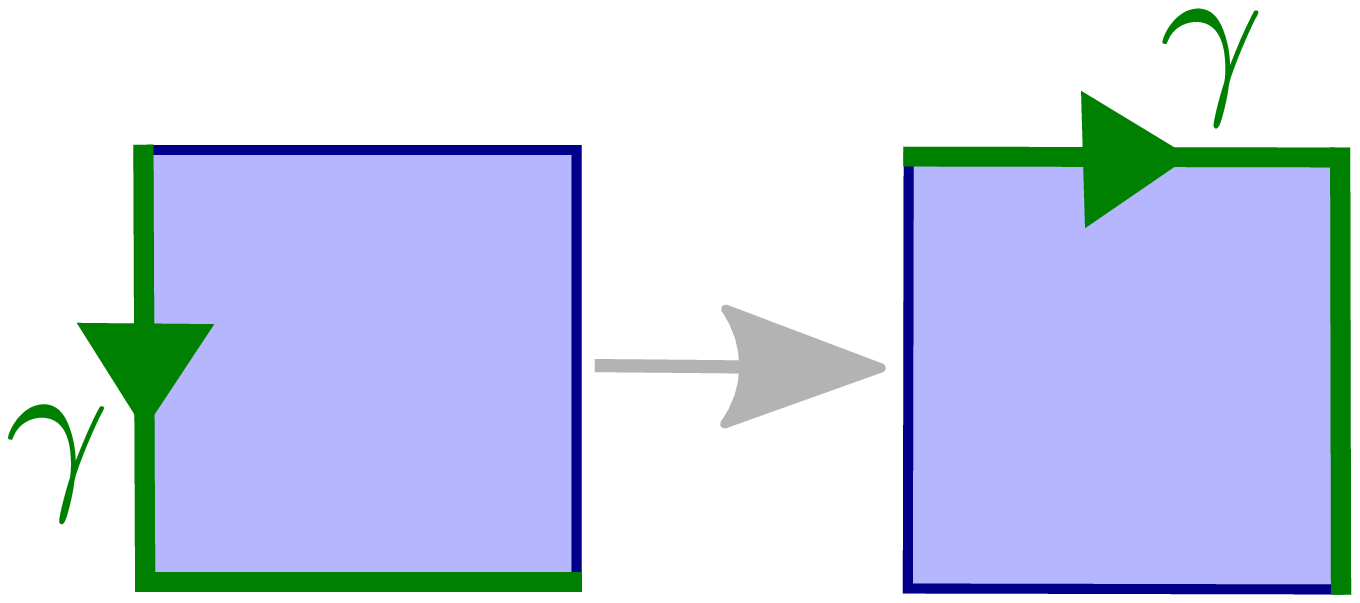}
\end{wrapfigure}

Let $W$ be a wall in $X$. Consider an edge loop $\gamma$ in $X$. $X$ is CAT(0), so in particular it is simply connected, thus $\gamma$ can be homotoped down to a constant loop by a sequence of moves that add/remove backtracks or push a subpath of $\gamma$ across a square in $X$ as shown (see \cite[I.8A.4]{nonpos}). The parity of the number of times $\gamma$ crosses $W$ is preserved by these moves, so $\gamma$ must originally have crossed $W$ an even number of times. 

 Consider an equivalence relation similar to that of Definition \ref{twoside}, but defined on all vertices of $X$: $x\sim y$ if $x$ can be joined to $y$ by an edge path that never crosses $W$. Write $[x]$ for the equivalence class of $x$. Let $e$ be an edge with endpoints $x,y$ that is dual to $W$. Note that $[x]\neq[y]$ because loops containing $e$ cross $W$ an even number of times. Clearly any vertex in $X$ is equivalent to one in $N(W)$, so by Proposition \ref{1or2sides}, $[x]$ and $[y]$ are the only classes.
 
Let $X(x):=\{z\in X- W\,|\,d(z,[x])<d(z,[y])\}$ - this is the union of all cubes containing only vertices in $[x]$ plus, for every cube with a midcube in $W$, the open half cube that has a vertex in $[x]$. Define $X(y)$ similarly. By construction $X(x)$ and $X(y)$ are open, disjoint and their union is $X- W$. \qed

The closures of $X(x)$ and $X(y)$ above are called \textit{half-spaces}. Each half-space is a convex subcomplex of $\dot{X}$ (where the dot denotes barycentric subdivision), because if a geodesic joining two points in the same half-space crosses into the other half-space this would create a geodesic between points in $W$ that leaves $W$, contradicting the convexity of $W\subset X$.

\begin{prop}\label{edgewall}
Let $e_1,e_2$ be edges incident at a vertex $x$ in a CAT(0) cube complex $X$. If $W(e_1)=W(e_2)$ then $e_1=e_2$.
\end{prop}
\begin{proof}
Let $W=W(e_1)=W(e_2)$. $W$ is a closed convex subspace of $X$, so there is a well-defined closest point projection $p:X\to W$. If $y_1$ is the midpoint of $e_1$, then $d(x,y_1)=1/2$, and no other point of $W$ can be closer to $x$ (the open ball of radius $1/2$ about $x$ is contained in the cubes incident at $x$ and doesn't touch any wall), so $p(x)=y_1$. But similarly if $y_2$ is the midpoint of $e_2$ then $p(x)=y_2$; hence $y_1=y_2$ and $e_1=e_2$.
\end{proof}
\begin{prop}\label{edgegeodesic}
Let $x,y$ be vertices in a CAT(0) cube complex $X$. An edge path between $x$ and $y$ will be of minimal length if and only if it only crosses walls that separate $x$ and $y$ and it crosses each of these once.
\end{prop}
\begin{proof}
An edge path between $x$ and $y$ must cross every wall that separates $x$ and $y$, so any edge path from $x$ to $y$ that only crosses these walls, and crosses each of them once, is necessarily of minimal length.
 
 For the converse implication, let $\gamma$ be a shortest edge path from $x$ to $y$ and suppose for contradiction that it crosses some wall twice. Say $e_1,...,e_n$ are the edges of a subpath of $\gamma$ with $W(e_1),...,W(e_{n-1})$ distinct and $W=W(e_1)=W(e_n)$.
 
 Suppose first that $e_2,...,e_{n-1}$ all lie in $N(W)$. Then there must be a square containing $e_2$ that intersects $W$; and if $e'_1$ is the edge of this square that crosses $W$ and meets $e_2$ at the same vertex as $e_1$, then Proposition \ref{edgewall} implies that $e_1=e'_1$. But then we can replace $e_1,e_2$ with edges $f_1, f_2$ going the other way round the square. The result is that we have pushed the first crossing of $W$ further along the subpath, whilst preserving the length of $\gamma$. Repeating this we can replace the subpath with $f_1,...,f_{n-1},e_n$ where $W=W(f_{n-1})$. But then,  by Proposition \ref{edgewall}, $f_{n-1},e_n$ must be a backtrack, contradicting the minimality of $\gamma$.
 
 Suppose now that $e_i$ is the first edge in the subpath that leaves $N(W)$. Note that $N(W)$ is the $\frac{1}{2}$-neighbourhood of $W$, so it is convex. The first half of $e_i$ is a geodesic $\eta$ between $N(W)$ and $W(e_i)$; if $N(W)$ and $W(e_i)$ intersect then we can form a geodesic triangle between a point of the intersection and $\eta$, but the angles at either end of $\eta$ are at least $\pi/2$, contradicting $X$ being CAT(0). Thus $N(W)\cap W(e_i)=\emptyset$ and $W(e_i)$ is disjoint from $W$. But then our subpath must recross $W(e_i)$ before it can get back to $W$, which is a final contradiction.
 \end{proof}
 
\begin{prop}\label{intersectwalls}(Helly's Theorem for walls of CAT(0) cube complexes)\cite[Lemma 13.13]{HW}\\
If $W_1,...,W_n$ are pairwise intersecting walls in a CAT(0) cube complex $X$, then there is a cube $C$ with $C\cap W_1\cap...\cap W_n\neq\emptyset$.
\end{prop}
\newpage

\section{Special cube complexes}\label{sec:special}

\begin{defn}
	A cube complex is \textit{simple} if the link of each vertex is a simplicial complex. In particular, a pair of edges meeting at a vertex $x$ cannot form two different corners of squares at $x$, as this would be a double edge in the link of $x$. 
\end{defn}
\begin{defn}\label{special}($C$-Special cube complex)\\
A simple cube complex $X$ is \textit{$C$-special} if:
\begin{enumerate}[(1)]
\item $W(e_1)\neq W(e_2)$ for any two distinct edges $e_1,e_2$ incident at a vertex $x$.
\item Given distinct edges $e_1,e_2$ incident at a vertex $x$ with $W(e_1)\cap W(e_2)\neq\emptyset$, we have that $e_1,e_2$ form the corner of a square in $X$.
\item $X^{(1)}$ is a bipartite graph.
\end{enumerate}
We use the definition of $C$-special from \cite[Definition 3.2]{HW} rather than the definition of special simply because it is easier to state (it avoids notions of oriented edges and oriented walls). The two definitions are closely related: any $C$-special cube complex is also special, and any compact special cube complex has a finite-sheeted $C$-special cover \cite[Proposition 3.10]{HW}. The results in this section hold for either definition.
\end{defn}

\begin{remk}
	A $C$-special (or special) cube complex might not be NPC, but it can be made NPC in a unique way by attaching cubes of dimension at least three \cite[Lemma 3.13]{HW}.
\end{remk}

\begin{remk}
We can also think of the definition of $C$-special cube complex as ruling out certain behaviours of hyperplanes. Property (1) of the definition rules out self-intersections and self-osculations of hyperplanes as illustrated below (hyperplanes in red, edges of the cube complex in blue). Property (2) of the definition forbids inter-osculations of hyperplanes. Property (3) ensures that the fundamental group embeds in a right-angled Coxeter group (hence the $C$ in $C$-special), this property is less important and indeed is absent from the definition of special.
\end{remk}

\includegraphics[width=0.55\textwidth, clip=true ]{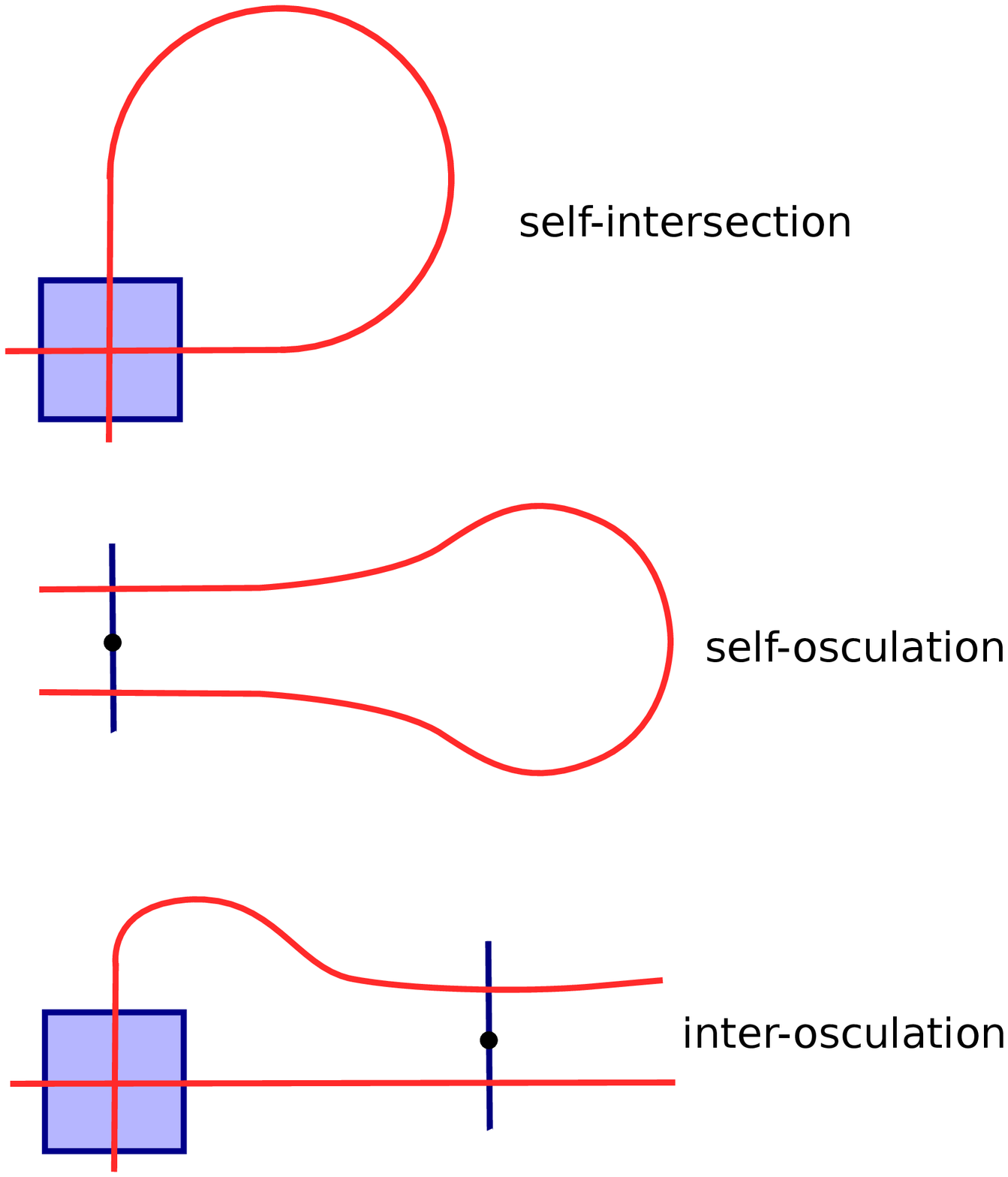}

Being $C$-special grants a cube complex some of the nice geometry of CAT(0) cube complexes, but without the strong restriction of being simply connected. We now state an assortment of results about $C$-special cube complexes that we will need later. Let's start with two easy propositions.

\begin{prop}\label{coverspecial}
Any covering of a $C$-special cube complex is $C$-special.
\end{prop}
\begin{proof}
Suppose $\phi:Y\to X$ is a covering of cube complexes with $X$ $C$-special. Being simple is a local condition about links of vertices, so the simplicity of $Y$ follows from the simplicity of $X$. We need to check that $Y$ satisfies properties (1)-(3) of Definition \ref{special}. $\phi(W(e))=W(f(e))$, so $W(e_1)=W(e_2)$ implies $W(f(e_1))=f(W(e_1))=f(W(e_2))=W(f(e_2))$, therefore $Y$ satisfies (1). Similarly $W(e_1)\cap W(e_2)\neq\emptyset$ implies $W(f(e_1))\cap W(f(e_2))\neq\emptyset$, so (2) holds. Finally (3) is true by taking the two vertex classes in $Y$ to be preimages of the vertex classes in $X$.
\end{proof}

\begin{prop}\label{freeprodspecial}
If $X_1,...,X_n$ are $C$-special cube complexes and $x_i\in X_i$ are choices of base vertex, then the cube complex obtained from $\sqcup_i X_i$ by identifying $x_1,...,x_n$ is also $C$-special.
\end{prop}
\begin{proof}
This follows immediately from the definition of $C$-special cube complex because each wall will be contained in a single $X_i$ factor.
\end{proof}

A compact cube complex is \textit{virtually special} if it has a finite-sheeted cover that is special (or equivalently a finite-sheeted cover that is $C$-special). A group is \textit{virtually special} if it has a finite index subgroup that is the fundamental group of a compact special cube complex. We will now state some powerful theorems showing that hyperbolic groups enjoy some strong properties if they are fundamental groups of virtually special cube complexes. The first of these theorems, due to Wise, characterises the fundamental groups of compact NPC virtually special cube complexes using the following quasi-convex hierarchy of hyperbolic groups.

\begin{defn}\label{QVH}
Let $\mathcal{QVH}$ denote the smallest class of hyperbolic groups that is closed under the following operations:
\begin{enumerate}[1.]
\item $1\in\mathcal{QVH}$.
\item If $G=A*_B C$ and $A,C\in\mathcal{QVH}$ and $B$ is finitely generated and quasi-convex in $G$, then $G\in\mathcal{QVH}$.
\item If $G=A*_B $ and $A\in\mathcal{QVH}$ and $B$ is finitely generated and quasi-convex in $G$, then $G\in\mathcal{QVH}$.
\item Let $H<G$ with $|G:H|<\infty$. If $H\in\mathcal{QVH}$ then $G\in\mathcal{QVH}$.
\end{enumerate}
\end{defn}

\begin{thm}(Wise, 2011)\cite[Theorem 13.3]{qconvexh}\label{qvhspecial}\\
A torsion-free hyperbolic group is in $\mathcal{QVH}$ if and only if it is the fundamental group of a compact virtually special cube complex.
\end{thm}

In particular, this implies that a torsion-free hyperbolic group is virtually special if and only if it is the fundamental group of a compact virtually special cube complex.
Theorem \ref{qvhspecial} is even more powerful when used in combination with the following theorem, which appeared in Haglund and Wise's original paper on special cube complexes. This theorem concerns the important notion of subgroup separability, which is defined as follows.

\begin{defn}(Separable subgroup)\\
	Let $G$ be a group. A subgroup $H$ of $G$ is \textit{separable} (in $G$) if for every $g\in G-H$ there is a homomorphism $\phi:G\to F$ such that $F$ is finite and $\phi(g)\notin\phi(H)$.
\end{defn}

\begin{thm}(Haglund-Wise, 2008)\cite[Theorems 1.3 and 1.4]{HW}\label{subgroupsep}\\
Let $X$ be a compact NPC cube complex with $\pi_1(X)$ hyperbolic. Then $X$ is virtually special if and only if every quasi-convex subgroup of $\pi_1(X)$ is separable.
\end{thm}

\begin{cor}\label{bothspecial}
Let $X$ and $Y$ be compact NPC cube complexes with $\pi_1(X)\cong\pi_1(Y)$ hyperbolic. Then $X$ is virtually special if and only if $Y$ is virtually special.
\end{cor}

These theorems will be used in Section \ref{sec:gluing} in the following way. They allow us to glue together two compact NPC virtually special cube complexes $X_1$ and $X_2$ along locally convex subcomplexes $Z_1\cong Z_2$ (locally convex in the usual geodesic space sense), to produce a larger virtually special cube complex $X$ - provided that $X_1, X_2$ and $X$ have hyperbolic fundamental groups.
We first argue that the composite $X$ is NPC, and then by considering universal covers we argue that $\pi_1(Z_1)$ is quasi-convex in $\pi_1(X)$ (these things will be explained more carefully later in the paper). We may then apply Theorem \ref{qvhspecial} to deduce that $\pi_1(X)=\pi_1(Y)$ for some compact NPC virtually special cube complex $Y$, and finally Corollary \ref{bothspecial} implies that $X$ itself is virtually special.

One problem we will face in Section \ref{sec:gluing} is that we might need to pass to finite covers of $X_1$ and $X_2$ in order to get isomorphic subcomplexes that we can glue along. This will be facilitated by Theorem \ref{thm:command}, which is a consequence of Wise's Malnormal Special Quotient Theorem \cite[Theorem 12.3]{qconvexh}.

\begin{defn}(Almost malnormal subgroups)\\ \label{defn:almostMalnormal}
	Let $G$ be a group and $\mathcal{H}$ a collection of subgroups.
	The subgroups $\mathcal{H}$ are \emph{almost malnormal} (resp. \emph{malnormal}) if for $H, H' \in \mathcal{H}$ the intersection $H^g \cap H'$ being infinite (resp. non-trivial) implies that $H = H'$ and $g \in H$.
\end{defn}
\begin{thm}(Malnormal Special Quotient Theorem)\\ \label{thm:MSQT}
	Let $G$ be a virtually special hyperbolic group.
	Let $\{H_1, \ldots, H_m \}$ be an almost malnormal collection of quasi-convex subgroups.
	Then there exist finite index subgroups $\dot{H}_i \triangleleft H_i$, such that for any further finite index subgroups $\hat{H}_i < \dot{H}_i$, the quotient $\bar{G}=G / \llangle \hat{H}_1, \ldots, \hat{H}_m \rrangle$ is hyperbolic and virtually special.
\end{thm}

\begin{thm}\label{thm:command}
Let $G$ be a virtually special hyperbolic group.
Let $\{H_1, \ldots, H_m \}$ be an almost malnormal collection of quasi-convex subgroups.
Then there exist finite index subgroups $\dot{H}_i \triangleleft H_i$, such that for any further finite index subgroups $\hat{H}_i < \dot{H}_i$, with $\hat{H}_i\triangleleft H_i$, there exists a finite index subgroup $\hat{G}\triangleleft G$ with $\hat{G}\cap H_i=\hat{H}_i$.
\end{thm}
\begin{proof}
	Pick the subgroups $\dot{H}_i$ according to Theorem \ref{thm:MSQT}. Then given further finite index subgroups $\hat{H}_i < \dot{H}_i$, with $\hat{H}_i\triangleleft H_i$, we may take the quotient $\bar{G}=G / \llangle \hat{H}_1, \ldots, \hat{H}_m \rrangle$, and Theorem \ref{thm:MSQT} tells us that it is virtually special. Fundamental groups of NPC cube complexes are torsion-free, as any torsion element would act on the universal cover with bounded orbits and fix a point by \cite[II.2.7]{nonpos}, hence $\bar{G}$ contains a finite index torsion-free subgroup $\hat{\bar{G}}\triangleleft \bar{G}$. Let $\hat{G}\triangleleft G$ be the preimage of $\hat{\bar{G}}$ with respect to the quotient map $G\to\bar{G}$. Each subgroup $H_i$ will have finite image in $\bar{G}$, so will intersect $\hat{\bar{G}}$ trivially, thus $\hat{G}\cap H_i=\ker(H_i\to\bar{G})$. By \cite[Theorem 1.1(1)]{Osin} $\ker(H_i\to\bar{G})=\hat{H}_i$ provided each $\dot{H}_i$ misses a given finite set in $H_i$ (which only depends on the $H_i$), so we are done by retrospectively modifying the $\dot{H}_i$.
\end{proof}

The last big theorem we state in this section, due to Agol, Groves and Manning, appears in the appendix of Agol's paper. We will use this theorem in the next section to take a quotient of the CAT(0) cube complex $X$ from Theorem \ref{mainthm} that makes the walls finite.

\begin{thm}(Agol-Groves-Manning, 2013)\cite[Theorem A.1]{VHC}\label{A1theorem}\\ 
Let $G$ be a hyperbolic group and $H<G$ a quasi-convex virtually special subgroup. Then for any $g\in G-H$, there is a hyperbolic group $\mathcal{G}$ and a homomorphism $\phi:G\to\mathcal{G}$ such that $\phi(g)\notin\phi(H)$ and $\phi(H)$ is finite.
\end{thm}

\section{Making walls finite}\label{sec:makefinite}

From now on let $G$ be a hyperbolic group acting properly and cocompactly on a CAT(0) cube complex $X$ as in Theorem \ref{mainthm}. The object of this section is to construct a quotient map $X\to\mathcal{X}$ such that walls in $\mathcal{X}$ are finite, and so that distinct walls in $X$ which are close together map to distinct walls in $\mathcal{X}$. The quotient complex $\mathcal{X}$ will be important for defining the local colouring data used in later sections. This section is based on \S4 of Agol's paper, but with considerably more detail added.

We may assume $X$ is unbounded since the theorem is trivial otherwise. For $x,y\in X$ we will use $[x,y]$ to denote the unique geodesic segment between them (with respect to the CAT(0) metric, as always). 

\begin{remk}\label{noswap}
By passing to the barycentric subdivision of $X$, we can assume that for every wall $W$ in $X$, $G_W$ does not exchange the sides of $W$, so $gW^\pm=W^\pm$ for all $g\in G_W$. By appropriate choice of labelling we can also assume that $gW^\pm=(gW)^\pm$ for all $g\in G$.
\end{remk}

\begin{prop}
$X$ is finite dimensional, locally finite and $\delta$-hyperbolic (for some $\delta$).
\end{prop}
\begin{proof}
$X$ is finite dimensional because $G$ acts cocompactly on it. Now suppose $X$ is not locally finite and that $x\in X$ is a vertex contained in infinitely many cubes. By cocompactness there is a cube $C$ and $A\subset G$ such that $A\cdot C$ is an infinite family of cubes each containing $x$; but then $C$ must have a vertex $x'$ such that $gx'=x$ for infinitely many $g\in A$, contradicting properness at x. Lastly, by \v{S}varc-Milnor, for any $x\in X$ the map $G\to X,\, g\mapsto gx$ is a quasi-isometry; hyperbolicity is a quasi-isometry invariant for geodesic spaces, and so $X$ is $\delta$-hyperbolic for some $\delta$.
\end{proof}

The remainder of this paper will be a proof of Theorem \ref{mainthm} by induction on $\dim X$ (the case $\dim X=0$ is trivial), so from now on assume that the theorem holds for lower dimensional cases - we will need this in the next two lemmas (in fact only for those lemmas).
 
These lemmas will use a few standard facts about hyperbolic and CAT(0) spaces, which we now recall.
\begin{enumerate}[(1)]
	\item For a geodesic $n$-gon in a $\delta$-hyperbolic space $Y$, each side is within the $(n-2)\delta$ neighbourhood of the union of the other sides (proof: subdivide the $n$-gon into triangles).
	\item If $C$ is a closed convex subspace of a CAT(0) space $Y$, then there is a well-defined closest point projection map $p:Y\to C$, and this map is distance non-increasing (see \cite[II.2.4]{nonpos}). In addition, $p$ commutes with any isometry of $Y$ that preserves $C$.
	\item If $C$ is a closed convex subspace of a CAT(0) $\delta$-hyperbolic space $Y$, $p:Y\to C$ the closest point projection map, and $A\subset Y$ another convex subspace with $p(A)$ unbounded, then $d(A,C)<2\delta$. Moreover, $N_{2\delta}(A)\cap C$ will be unbounded.
	
\begin{proof}
Let $x,y\in A$ and suppose $z\in[p(x),p(y)]\subset C$ with $d(z,p(x)),d(z,p(y))>4\delta$.	
	
\begin{minipage}{.4\textwidth}
	\centering
	\includegraphics[width=0.9\textwidth, clip=true ]{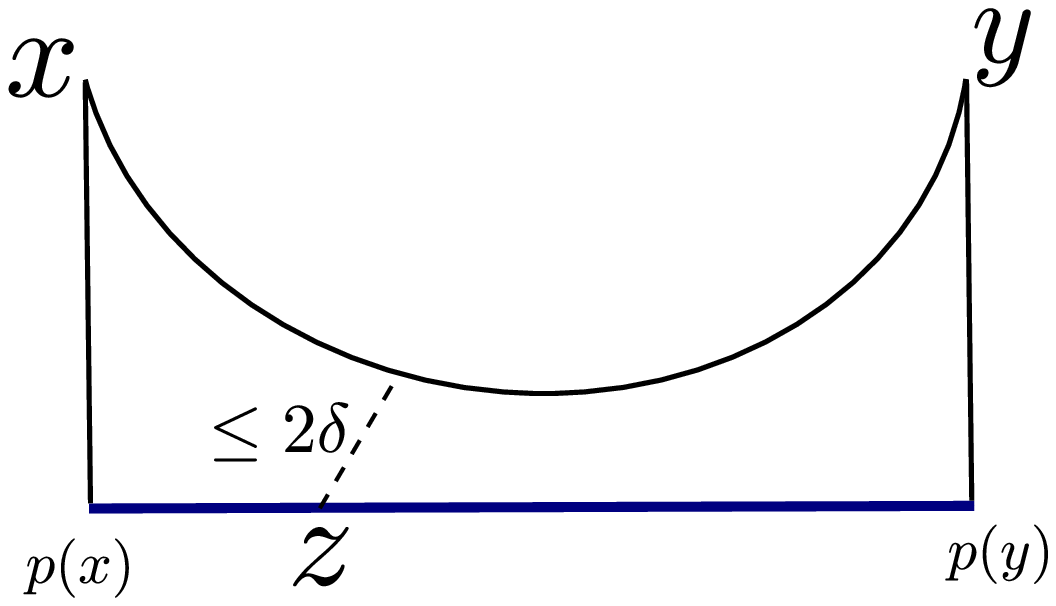}
\end{minipage}
\begin{minipage}{.5\textwidth}
	We must have $d(z,[x,p(x)])\geq2\delta$ as otherwise $z$ would be closer to $x$ than $p(x)$ is. Similarly  we must have $d(z,[y,p(y)])\geq2\delta$. By applying fact (1) to the geodesic quadrilateral shown, we deduce that $d(z,[x,y])<2\delta$. Thus $N_{2\delta}(A)\cap[p(x),p(y)]$ contains all of $[p(x),p(y)]$ except possibly the end segments of length $4\delta$.  As we can have $d(p(x),p(y))$ arbitrarily large, we see that $N_{2\delta}(A)\cap C$ must be unbounded.	
\end{minipage}
\end{proof}
\end{enumerate}

\begin{lem}\label{closewalls}
Either we can choose $W_1,...,W_m$ orbit representatives for the walls of $X$ such that $d(W_i,W_j)>3\delta$ for all $1\leq i<j\leq m$, or Theorem \ref{mainthm} holds.
\end{lem}
\begin{proof}	
 $G$ is hyperbolic, so contains an infinite order element $b$ \cite[$\Gamma$.2.22]{nonpos}. All isometries of $X$ are semi-simple \cite[II.6.10(2)]{nonpos}, so $b$ acts hyperbolically on $X$, and acts by translations on an axis $\gamma$ in $X$ (a geodesic line in $X$) \cite[II.6.8(1)]{nonpos}. Let $p:X\to\gamma$ be the closest point projection map to $\gamma$.
  
 First suppose there is a wall $W$, such that $p(gW)$ is unbounded for all $g\in G$. By fact (3), we know that $\gamma\cap N_{2\delta}(gW)$ is unbounded for all $g\in G$. Because we are in a CAT(0) space, and $\gamma$ and $gW$ are convex, we see that  $\gamma\cap N_{2\delta}(gW)$ is convex, so contains an infinite subinterval of $\gamma$. We deduce that there are only finitely many distinct translates $gW$, else infinitely many of them would be within $2\delta$ of some point on $\gamma$, contradicting local finiteness of $X$. This means that the stabiliser $G_{W}$ is finite index in $G$. But $G_{W}$ acts properly cocompactly on the CAT(0) cube complex $W$, so by the lower dimensional case of theorem \ref{mainthm} there is a finite index subgroup $G'<G_{W}$ acting freely on $W$ such that $W/G'$ is special. Then $G'$ also acts freely on $X$, for if $g\in G'$ fixed $x\in X$ then $g$ would also fix $p(x)$. Then $X/G'$ is virtually special by Corollary \ref{bothspecial}, and Theorem \ref{mainthm} holds by replacing $G'$ with a further finite index subgroup.
  
 Conversely, suppose that for every wall $W$ there exists $g\in G$ with $p(gW)$ bounded. Let $W_1,...,W_m$ be orbit representatives for the walls of $X$ such that $p(W_i)$ is bounded for each $i$. Each $p(W_i)$ is contained in a finite subinterval of $\gamma$, and $b$ acts as a translation along $\gamma$, so we may choose $n_1,...,n_m\in\mathbb{Z}$ such that $d(p(b^{n_i}W_i),p(b^{n_j}W_j))=d(b^{n_i}p(W_i),b^{n_j}p(W_j))>3\delta$ for all $1\leq i<j\leq m$. But $p$ is distance non-increasing by fact (2), thus $d(b^{n_i}W_i,b^{n_j}W_j)>3\delta$ for all $1\leq i<j\leq m$, as required.
\end{proof}

Henceforth we will assume that we are in the first scenario of Lemma \ref{closewalls}, so for the remainder of this section let $W_1,...,W_m$ be orbit representatives for the walls of $X$ such that $d(W_i,W_j)>3\delta$ for all $1\leq i,j\leq m$. We are now ready for the main technical lemma of this section, which will be used in Lemma \ref{propmathcalX} to produce a quotient of $X$ with finite walls. Note that the homomorphism $\phi:G\to\mathcal{G}$ will be a product of homomorphisms obtained from Theorem \ref{A1theorem}, so in particular $\mathcal{G}$ might not be hyperbolic.

\begin{lem}\label{mathcalX}
For any $R>1$ large enough so that $G\cdot B=X$ for any $R$-ball $B$ in $X$, there exists a surjective homomorphism $\phi :G\to\mathcal{G}$ with kernel $K$ and $H_i \triangleleft G_{W_i}$ finite index such that
\begin{enumerate}[(1)]
\item $\phi(H_i)$ are all finite,
\item if $g\in G- H_i$ with $d(gW_i, W_i)\leq 2R$ then $\phi(g)\notin\phi(H_i)$,
\item $K$ is torsion-free (so acts freely on $X$).
\end{enumerate}
\end{lem}

The proof will use the following variant of the ping-pong lemma. Some of the arguments will be closely related to the results in \cite{Gitik}.

\begin{lem}(Ping-pong Lemma)\label{pingpong}\\
	Let $H$ be a group that acts on a set $Y$. If $Y_1,...,Y_n\subset Y$ and $H_1,...,H_n<H$ and $y_0\in Y-\cup_iY_i$ are such that $hY_j\subset Y_i$ and $hy_0\in Y_i$ whenever $1\neq h\in H_i$ and $j\neq i$, then $H$ splits as a free product $H=H_1*...*H_n$.
\end{lem}
\begin{proof}
	A product $h=h_1\cdots h_k$ with $1\neq h_i\in H_{m_i}$ and $m_i\neq m_{i+1}$ clearly maps $y_0$ into $Y_{m_1}$, and so is not the identity.
\end{proof}

\begin{breakproof}[Proof of Lemma \ref{mathcalX}]
As in fact (2) from earlier, let $p_i:X\to W_i$ be the closest point projection map to the wall $W_i$. As the $W_i$ are at least $3\delta$ apart from each other, fact (3) tells us that the images $p_j(W_i)$ for $i\neq j$ are all bounded. By Remark \ref{cocompactstabiliser} and induction on the lower dimensional cases of Theorem \ref{mainthm}, for each $i$ there exists $H_i<G_{W_i}$ finite index acting freely on $W_i$ with $W_i/H_i$ $C$-special. Define bounded subspaces
\begin{equation*}
A_i:=N_{14\delta+2R+1}(\cup_{j\neq i} p_i(W_j)).
\end{equation*}
Theorem \ref{subgroupsep} tells us in particular that $H_i$ is residually finite, so by replacing $H_i$ with a further finite index subgroup we can assume that $d(hA_i,A_i)>1$ for all $1\neq h\in H_i$. We can also assume that $H_i\triangleleft G_{W_i}$ by intersecting it with its finitely many conjugates, and $W_i/H_i$ will still be $C$-special by Proposition \ref{coverspecial}. Note that some $W_i$ might be finite and have $A_i=W_i$ - in these cases $H_i$ will be trivial.

Let $X_i:=p_i ^{-1}(W_i-A_i)$. Pick $x_0\in p_1(W_2)\subset A_1$ and note that $x_0$ is not in any of the $X_i$. The next part of the proof does ping-pong with $x_0, H_i$ and $X_i$ to prove that we get a free splitting $H:=\langle H_1,...,H_m\rangle\cong H_1\ast...\ast H_m$. By ignoring the $i$ for which $H_i$ is trivial we can assume that the sets $W_i-A_i$ and $X_i$ are non-empty.

\begin{claim}
$p_i(X_j)\subset A_i$ for $j\neq i$.
\end{claim}

\begin{claimproof}
Let $x\in X_j$ and suppose for contradiction that $p_i(x)\notin A_i$.
 
We have the geodesic pentagon shown, where $y$ is any point in $p_j(W_i)$ and $z$ is a point on $[p_i(x),p_i(y)]\cap A_i$. We have defined $A_i$ to include a $14\delta$ buffer zone around $\cup_{j\neq i} p_i(W_j)$, and $p_i(x)\notin A_i$, therefore we can choose $z$ to satisfy $d(z,p_i(x)),d(z,p_i(W_j))>7\delta$. By fact (1) from earlier, $z$ is within $3\delta$ of one of the other sides of the pentagon, we now check each of these four sides in turn:
\begin{minipage}{.5\textwidth}

\begin{enumerate}[1.]
\item If $z'\in [x,p_j(x)]$ then $p_j(z)=p_j(x)\notin A_j$, and so $d(p_j(z),p_j(z'))\geq d(p_j(W_i),p_j(x))$, which is greater than $3\delta$ because of the buffer zone in $A_j$. But $p$ is distance non-increasing, so $d(z,z')>3\delta$.
\item $d(z, W_j)\geq d(W_i,W_j)>3\delta$ by choice of the walls $W_k$.
\item If $z'\in[y,p_i(y)]$ then
\begin{align*}
d(y,z)&\leq d(y,z')+d(z',z)\\
&=d(y,p_i(y))-d(p_i(y),z')+d(z',z)\\
&\leq d(y,p_i(y))-d(z,p_i(y))+2d(z,z')\\
&\leq d(y,p_i(y))-7\delta+2d(z,z')\\
&\leq d(y,z)-7\delta+2d(z,z'),
\end{align*}
(the last inequality by definition of $p_i$). This implies $d(z,z')>3\delta$.
\item The same argument as 3. shows that $z$ cannot be within $3\delta$ of $[x,p_i(x)]$.
\end{enumerate}
\end{minipage}
\begin{minipage}{.5\textwidth}
\centering
\includegraphics[width=.9\textwidth,clip=true]{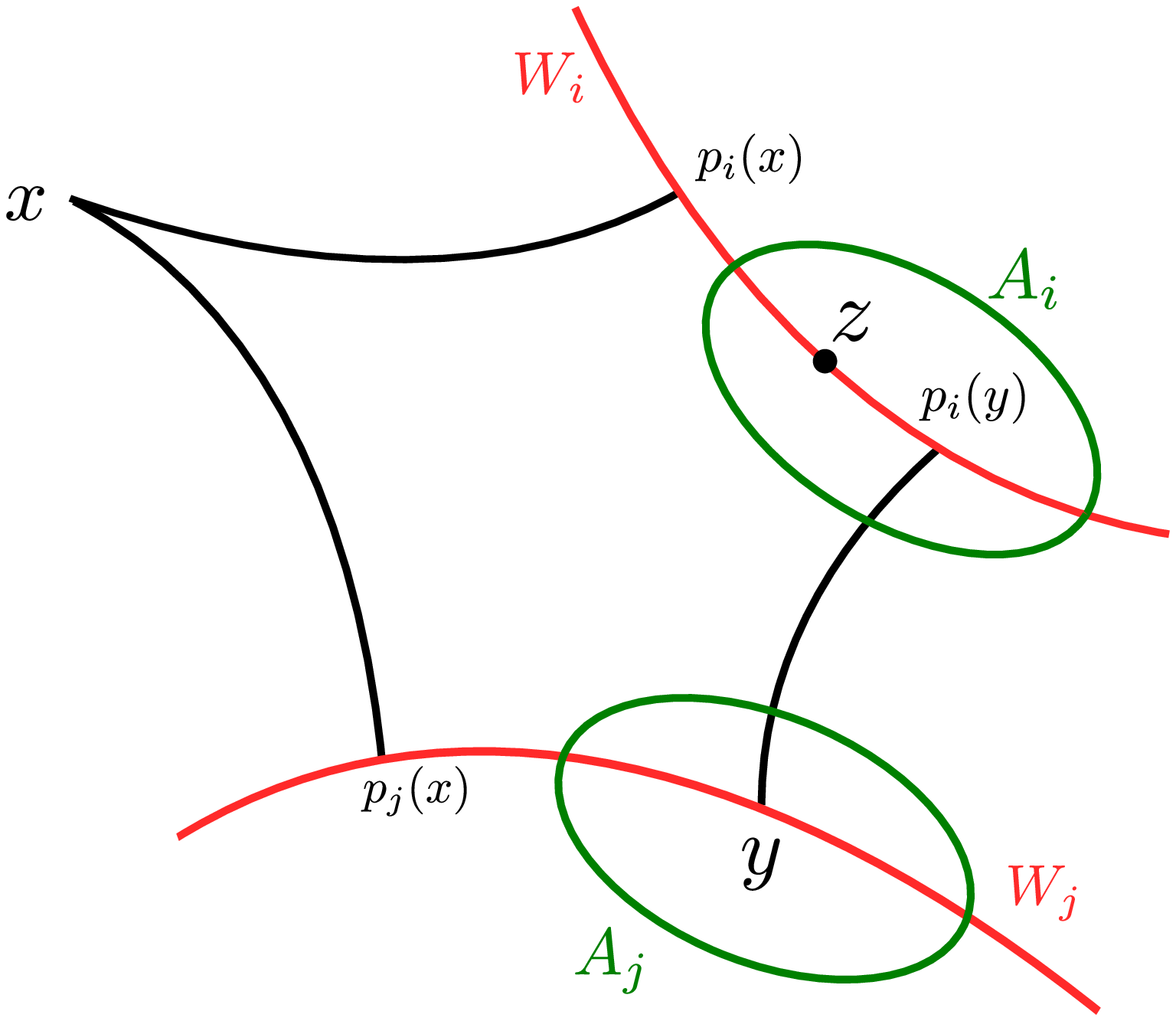}
\end{minipage}

We conclude that $z$ is not within $3\delta$ of one of the other sides of the pentagon, contradicting fact (1). The claim follows.
\end{claimproof}

\begin{claim}
For $j\neq i$ and $1\neq h\in H_i$ we have $hX_j\subset X_i$, $h W_j\subset X_i$ and $hx_0\in X_i$.
\end{claim}

\begin{claimproof}
Let $x\in X_j\cup W_j$. By the previous claim we have $p_i(x)\in A_i$, so $p_i(hx)=hp_i(x)\in hA_i$, hence $p_i(hx)\notin A_i$ and so $hx\in X_i$. Additionally, $p_i(hx_0)=hp_i(x_0)\in hA_i$ and so $hx_0\in X_i$.
\end{claimproof}

This last claim allows us to do ping-pong, as in Lemma \ref{pingpong}, to obtain the desired splitting $H\cong H_1\ast...\ast H_m$.

\begin{claim}
$H<G$ is quasi-convex.
\end{claim}

\begin{claimproof}
$G\to X, g\mapsto gx_0$ is a quasi-isometry, so it suffices to show that $H\cdot x_0$ is quasi-convex in $X$.

Let $h=h_1h_2...h_k$ with $1\neq h_i\in H_{n_i}$ and $n_i\neq n_{i+1}$. Put $g_i=h_1...h_i$ and $g_0=1$. Our strategy will be to show that all points on the geodesic $[x_0,hx_0]$ are close to one of the walls $g_i W_{n_i}$ for $1\leq i\leq k$. 

First we consider projections to such a wall. Let $q_i:X\to g_iW_{n_i}$ be the closest point projection map to $g_iW_{n_i}$. For $x\in X$, $q_i(x)$ is the closest point on $g_iW_{n_i}$ to $x$, left multiplying by $g_i^{-1}$ then tells us that $g_i^{-1}q_i(x)$ is the closest point on $W_{n_i}$ to $g_i^{-1}x$, so $g_i^{-1}q_i(x)=p_{n_i}(g_i^{-1}x)$. Therefore \begin{equation}\label{qpconj}
q_i=g_ip_{n_i}g_i^{-1}.
\end{equation}
We can then compute for $1\leq i<k$,
\begin{align}
q_i(h x_0)&=g_i p_{n_i}(g_i ^{-1} h x_0)\nonumber\\
&=g_i p_{n_i}(h_{i+1}...h_k x_0)\nonumber\\
&\in g_i p_{n_i}( X_{n_{i+1}})\hspace{2cm}&\text{by the second claim,}\nonumber\\
&\subset g_i A_{n_i}\hspace{2cm}&\text{by the first claim.}\label{projection1}
\end{align}
Similarly,
\begin{equation}\label{projection1'}
q_k(hx_0)=hp_{n_k}(x_0)\in h A_{n_k}.
\end{equation}
Next observe that $g_iW_{n_i}=g_{i-1}W_{n_i}$, and so analogously to (\ref{qpconj}) we have $q_i=g_{i-1}p_{n_i}g_{i-1}^{-1}$ ($1\leq i\leq k$). We then compute for $1<i\leq k$,
\begin{align}
q_i(x_0)&=g_{i-1} p_{n_i}(g_{i-1} ^{-1} x_0)\nonumber\\
&=g_{i-1} p_{n_i}(h^{-1}_{i-1}...h^{-1}_1 x_0)\nonumber\\
&\in g_{i-1} p_{n_i}( X_{n_{i-1}})\hspace{2cm}&\text{by the second claim,}\nonumber\\
&\subset g_{i-1} A_{n_i}\hspace{2cm}&\text{by the first claim.}\label{projection2}
\end{align}
And similarly
\begin{equation}\label{projection2'}
q_1(x_0)=p_{n_1}(x_0)\in A_{n_1}.
\end{equation}

We now consider the concatenation of geodesics joining the following points pairwise in order.
\begin{equation*}
x_0,q_1(x_0), q_1(hx_0),q_2(x_0), q_2(hx_0),...,q_k(x_0),q_k(hx_0), hx_0
\end{equation*}

Call this path $\gamma$, and refer to the above points as the vertices of $\gamma$. Recalling that $x_0\in A_1$, we can bound every other gap between consecutive vertices as follows.
\begin{equation*}
D:=\text{diam}(\cup A_j)\geq
\begin{cases}
d(x_0,q_1(x_0)),&\text{by (\ref{projection2'})}\\
d(q_i(hx_0),q_{i+1}(x_0)),&\text{for $1\leq i< k$, by (\ref{projection1}) and (\ref{projection2})}\\
d(q_k(hx_0), hx_0),&\text{by (\ref{projection1'})}
\end{cases}
\end{equation*}

The other gaps between consecutive vertices are spanned by segments $\gamma_i:=[q_i(x_0),q_i(hx_0)]\subset g_iW_{n_i}$. Since $H_{n_i}$ acts cocompactly on $W_{n_i}$ and $g_i\in H$, we deduce that each $\gamma_i$ is contained within $N_M(H\cdot x_0)$ for some constant $M$ that is independent of $h$. Hence $\gamma\subset N_{M+D}(H\cdot x_0)$.

To complete the proof of the claim it remains to show that $\sigma:=[x_0,hx_0]\subset N_L(\gamma)$ for some constant $L$ that is independent of $h$.

\begin{minipage}{.3\textwidth}
Consider $z\in \gamma_i$ at least $5\delta$ away from the endpoints of $\gamma_i$. By fact (1), $z$ is within $2\delta$ of one of the other sides of the geodesic quadrilateral shown, so it must be within $2\delta$ of $\sigma$ - otherwise it contradicts the definition of closest point projection. Therefore $\gamma_i\subset N_{7\delta}(\sigma)$ and $\gamma\subset N_{D+7\delta}(\sigma)$.
\end{minipage}
\begin{minipage}{.6\textwidth}
\centering
\includegraphics[width=.6\textwidth,clip=true]{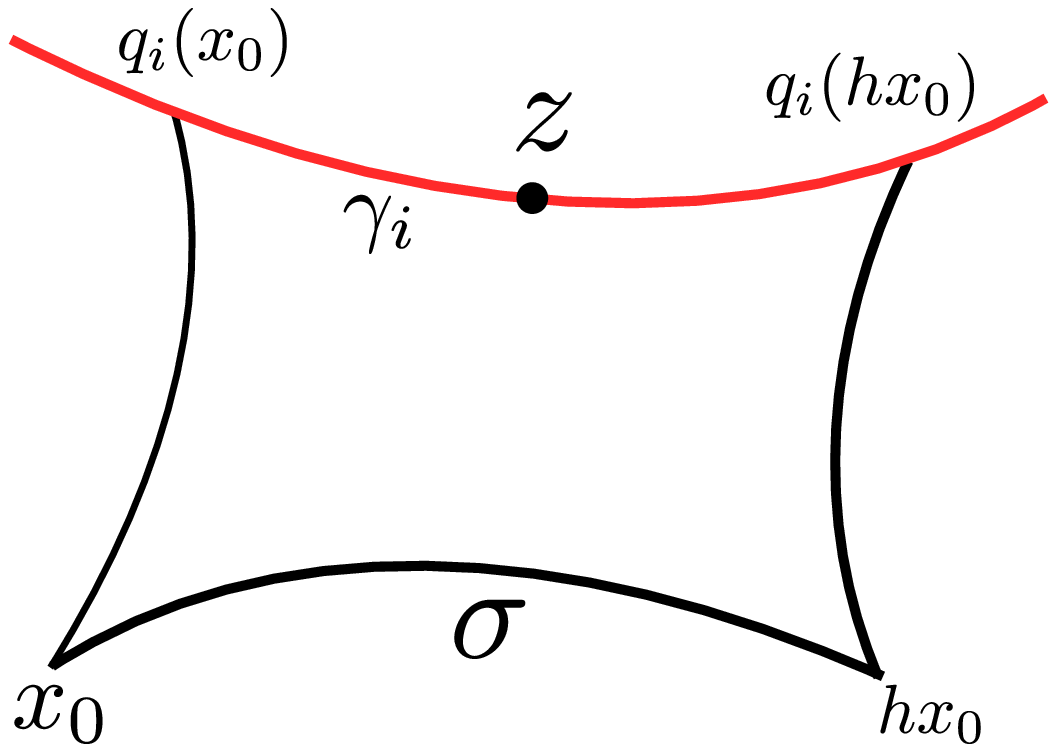}
\end{minipage}

Finally note that projection from $\gamma$ to $\sigma$ is continuous; and, as the paths share endpoints $x_0$ and $hx_0$, the image is the whole of $\sigma$. So in fact $\sigma\subset N_{D+7\delta}(\gamma)$.
\end{claimproof}

By Proposition \ref{freeprodspecial}, $H\cong H_1\ast...\ast H_m$ is the fundamental group of a $C$-special cube complex. By taking direct products of the homomorphisms in \ref{A1theorem}, we deduce that for any finite $A\subset G-H$ there is a quotient homomorphism $\phi:G\to\mathcal{G}$ such that $\phi(A)\cap\phi(H)=\emptyset$ and $\phi(H)$ is finite. We now show that conclusions (1)-(3) of the lemma can be satisfied by a certain choice of $A$.

\begin{enumerate}[(1)]
\item $\phi(H_i)<\phi(H)$ so must be finite.
\item For each $i$ the collection of double cosets
\begin{equation*}
\mathcal{A}_i:=\{H_i gH_i\,|\,g\in G,\,d(gW_i,W_i)\leq2R\}-\{H_i\}
\end{equation*}
is finite. To see this, fix $y\in W_i$ and consider $g\in G$ with $d(gW_i,W_i)\leq2R$ - say $x,x'\in W_i$ satisfy $d(gx',x)\leq2R$. Suppose that $Q>0$ with $W_i\subset H_i\cdot B_Q(y)$. Now pick $h,h'\in H_i$ so that $d(hx,y),d(h'y,x')<Q$; then $d(hgh'y,y)\leq d(hgh'y,hgx')+d(hgx',hx)+d(hx,y)<2Q+2R$. The finiteness of $\mathcal{A}_i$ then follows because $X$ is locally finite and the action of $G$ is proper.

We now claim that $g\in G- H_i$ with $d(gW_i,W_i)\leq 2R$ cannot have $g\in H$. Indeed, if $g=h_1 h_2...h_k$ with $h_j\in H_{n_j}$ and $n_j\neq n_{j+1}$, then by the second claim we see that $g W_i\subset X_{n_1}$ and $p_{n_1}(gW_i)\subset W_{n_1}-A_{n_1}$. Removing $h_1$ if necessary we can assume that $n_1\neq i$. But there exists $x\in g W_i$ with $d(x,W_i)\leq 2R$, and fact (2) from earlier implies that $d(p_{n_1}(x),p_{n_1}(W_i))\leq 2R$ and so $p_{n_1}(x)\in A_{n_1}$, a contradiction.

Given these two facts, we can ensure that $A$ contains representatives for all of the double cosets in the $\mathcal{A}_i$, and this ensures that $\phi(g)\notin\phi(H_i)$ for any $g\in G- H_i$ with $d(gW_i,W_i)\leq 2R$.
\item If $g\in G$ is a torsion element then by \cite[II.2.8]{nonpos} it has a fixed point $x\in X$. By assumption of the lemma, there exists $k\in G$ with $kx\in B_R(x_0)$. Then $d(kgk^{-1}x_0,x_0)<2R$. Therefore there is a finite set $\mathcal{T}$ of representatives for conjugacy classes of torsion elements in $G$. Each $H_i$ is torsion-free because it acts freely on $W_i$, so $H$ is also torsion-free and $H\cap\mathcal{T}=\emptyset$. Adding $\mathcal{T}$ to $A$ will ensure that $K$ is torsion-free.
\end{enumerate}

\end{breakproof}

The point of Lemma \ref{mathcalX} is that it allows us to define the following quotient complex.
\begin{defn}(Quotient Complex $\mathcal{X}$)\label{mathcalXdef}\\
As a result of Lemma \ref{mathcalX}, we can define the NPC cube complex $\mathcal{X} :=X/K$. The value of $R$ we use will be some constant large enough to satisfy the cocompactness condition of the lemma, and we also require $R\geq \delta+2\sqrt{\dim{X}}$ (this inequality will be demystified in section 8). The metric on $\mathcal{X}$ will be denoted $d$, the same as for $X$.
\end{defn}

As was the aim of this section, this quotient complex satisfies the following properties.
\begin{lem}(Properties of $\mathcal{X}$)\label{propmathcalX}
	\begin{enumerate}[(1)]
		\item There are natural cocompact actions of $G$ and $\mathcal{G}$ on $\mathcal{X}$.
		\item All walls of $\mathcal{X}$ are finite.
		\item For any wall $W$ in $X$, the $R$-neighbourhood $N_R(W)$ quotiented by $K\cap G_W$ embeds in $\mathcal{X}$. In particular this implies that all walls of $\mathcal{X}$ are embedded, and that distinct walls in $X$ which are less than $R$ apart map to distinct walls in $\mathcal{X}$.
	\end{enumerate}	
\end{lem}
\begin{proof}
	(1) holds because $K$ is normal in $G$ and $G$ acts cocompactly on $X$. By Lemma \ref{mathcalX}(1) we know that, for any $g\in G$, $K\cap gG_{W_i}g^{-1}$ has finite index in $gG_{W_i}g^{-1}=G_{gW_i}$ and so acts cocompactly on $gW_i$ - property (2) follows. Lemma \ref{mathcalX}(2) tells us that the $R$-neighbourhood $N_R(W_i)$ quotiented by $K\cap H_i=K\cap G_{W_i}$ embeds in $\mathcal{X}$ - property (3) follows by considering translates of the $W_i$ and conjugates of the $H_i$.
\end{proof}

To finish the section we introduce some notation.

\textbf{Notation:}\,
The quotient map $m:X\to\mathcal{X}$ will send a vertex $x$ (resp. an edge $e$ and wall $W$) to a vertex $\bar{x}$ (resp. an edge $\bar{e}$ and wall $\overline{W}$). And $\overline{W}(\bar{e})$ will denote the wall dual to $\bar{e}$. By Remark \ref{noswap} we know that walls in $\mathcal{X}$ are two-sided and that no element of $\mathcal{G}$ can exchange the sides of any wall. We can define $\overline{W}^\pm:=\overline{(W^\pm)}$.

\section{Invariant colouring measures}\label{sec:invariantcolouring}

This section is preparatory. More specifically, before we can start colouring walls in the next section, we need to establish some theory about colouring graphs. This section is essentially the same as \S5 of Agol's paper, but with a little more detail regarding weak$^*$ convergence.

\begin{defn} (Colourings)\\\label{colourings}
An \textit{$n$-colouring} of a graph $\Gamma$ is a map $c:V(\Gamma)\to [n]:=\{1,...,n\}$ such that $c(v_1)\neq c(v_2)$ whenever $\{v_1,v_2\}\in E(\Gamma)$. Let $C_n(\Gamma)$ denote the set of $n$-colourings. If vertex degrees are bounded by $k$ then it is clear that $C_{k+1}(\Gamma)\neq\emptyset$.

Suppose a group $H$ acts on $\Gamma$. Then we have an action of $H$ on $C_n(\Gamma)$ by $h: c\mapsto c\circ h^{-1}$.
\end{defn}

\begin{defn}(Colourings as a measurable space)\\
Consider $C_n(\Gamma)$ as a closed subspace of $[n]^{V(\Gamma)}$ with the product topology. We will consider $[n]^{V(\Gamma)}$ as a measurable space with $\sigma$-algebra generated by the sets $A_{v,j}:=\{f\in [n]^{V(\Gamma)}|\,f(v)=j\}$ for $v\in V(\Gamma)$ and $j\in[n]$ - note that if $V(\Gamma)$ is countable then this is also the $\sigma$-algebra generated by the open subsets of $[n]^{V(\Gamma)}$. The action of $H$ on $C_n(\Gamma)$ extends naturally to $[n]^{V(\Gamma)}$ by $h:f\mapsto f\circ h^{-1}$. These are measurable functions so the family of measurable subsets of $[n]^{V(\Gamma)}$ is $H$-invariant. 
\end{defn}

\begin{thm}\label{measure} Suppose a group $H$ acts cocompactly on a countable graph $\Gamma$ with vertex degrees bounded by $k$. Then there exists an $H$-invariant probability measure $\mu$ on $C_{k+1}(\Gamma)$.
\end{thm}
\begin{proof}
Let $M_H(n)$ denote the set of $H$-invariant probability measures on $[n]^{V(\Gamma)}$. Note that $M_H(n)$ is non-empty because it contains the measure $\mu_n$ which is the product of uniform measures on $[n]$. We have that $C_{k+1}(\Gamma)\subset[k+1]^{V(\Gamma)}$, so our task is to find $\mu\in M_H(k+1)$ with $\mu(C_{k+1}(\Gamma))=1$ (thus making $\mu$ an $H$-invariant probability measure on $C_{k+1}(\Gamma)$).

For an edge $e=\{v_1,v_2\}\in E(\Gamma)$ let $B_e(n):=\{f\in[n]^{V(\Gamma)}|\,f(v_1)=f(v_2)\}$; it is clear that $[n]^{V(\Gamma)}$ splits as a disjoint union $[n]^{V(\Gamma)}=C_n(\Gamma)\sqcup\bigcup_{e\in E(\Gamma)}B_e(n)$, so our task reduces to finding $\mu\in M_H(k+1)$ with $\mu(B_e(k+1))=0$ for all $e\in E(\Gamma)$. Let $\{e_1,...,e_m\}\subset E(\Gamma)$ be a complete set of orbit representatives for the action of $H$ on $E(\Gamma)$; since $\nu(B_{he}(k+1))=\nu(hB_e(k+1))=\nu(B_e(k+1))$ for all $\nu\in M_H(k+1)$, it suffices to find $\mu\in M_H(k+1)$ with $\mu(B_{e_i}(k+1))=0$ for $i=1,...,m$.

For $\nu\in M_H(n)$ define $weight(\nu):=\sum_i \nu(B_{e_i}(n))$. We will use a limiting argument to construct $\mu\in M_H(k+1)$ with zero weight. As mentioned above, there is a $H$-invariant probability measure $\mu_n$ on $[n]^{V(\Gamma)}$ which is the product of uniform measures on $[n]$. Note that $\mu_n$ is the unique measure with $\mu_n(A_{v,j})=1/n$ for every $A_{v,j}$. $\mu_n(B_e(n))=\mu_n(\cup_j (A_{v,j}\cap A_{w,j}))=1/n$ for $e=\{v,w\}\in E(\Gamma)$, so $weight(\mu_n)=m/n$.

For $n>k+1$ define a map $p_n:[n]^{V(\Gamma)}\to[n-1]^{V(\Gamma)}$ by
\begin{equation*}
p_n(c)(v):=
\begin{cases}
c(v),&c(v)<n\\
\text{min}([n-1]-\{c(u)|\,\{u,v\}\in E(\Gamma)\}),&c(v)=n
\end{cases}
\end{equation*}
for $c\in[n]^{V(\Gamma)}$ and $v\in V(\Gamma)$. In other words $p_n$ changes the colour of each vertex coloured $n$ to the smallest colour not used by its neighbours, and leaves other vertices with the same colour. This is well-defined because vertex degrees are at most $k$. It is clear that $p_n$ is $H$-equivariant and continuous, so it induces a well-defined push-forward $p_{n\ast}:M_H(n)\to M_H(n-1)$ given by $p_{n\ast}(\nu)(A)=\nu(p_n^{-1}(A))$ for $\nu\in M_H(n)$ and $A\subset [n-1]^{V(\Gamma)}$ measurable. Furthermore, for $\{v_1,v_2\}\in E(\Gamma)$ and $c\in [n]^{V(\Gamma)}$, if $p_n(c)(v_1)=p_n(c)(v_2)$ then $c(v_1)=c(v_2)$. Therefore $p_n^{-1}(B_e(n-1))\subset B_e(n)$ for any edge $e$; consequently $weight(p_{n\ast}(\nu))\leq weight(\nu)$ for any $\nu\in M_H (n)$.

Now define $P_{n\ast}=p_{k+2\ast}\circ p_{k+3\ast}\circ\cdots\circ p_{n\ast}:M_H(n)\to M_H(k+1)$. We will then have that $weight(P_{n\ast}(\mu_n))\leq weight(\mu_n)=m/n\to0$ as $n\to\infty$.

 By Prokhorov's Theorem \cite{weakstar} the set of all probability measures on $[k+1]^{V(\Gamma)}$ is compact metrizable in the weak* topology (measures $(\nu_n)$ converge to $\nu$ in the weak* topology if and only if $\int\alpha\, d\nu_n\to\int\alpha\, d\nu$ for every continuous map $\alpha:[k+1]^{V(\Gamma)}\to\mathbb{R}$). $M_H(k+1)$ is a closed subspace with respect to this topology; to see this suppose a sequence $(\nu_n)$ in $M_H(k+1)$ converges to a measure $\nu$. Let $\Sigma$ be the algebra generated by the $A_{v,j}$ (the smallest family containing the $A_{v,j}$ that is closed under finite union and complementation), all sets in $\Sigma$ are clopen in $[k+1]^{V(\Gamma)}$ so by considering characteristic functions we have $\nu_n(A)\to\nu(A)$ for every $A\in\Sigma$. For $h\in H$ define $\nu_h$ by $\nu_h(A)=\nu(hA)$ for measurable sets $A$; for $A\in\Sigma$ we have $\nu_h(A)=\lim_{n\to\infty}\nu_n(hA)=\lim_{n\to\infty}\nu_n(A)=\nu(A)$, so by Caratheodory's Extension theorem we get $\nu=\nu_h$, thus $\nu\in M_H(k+1)$.
 
We conclude that $P_{n\ast}(\mu_n)$ has a convergent subsequence converging to some $\mu\in M_H(k+1)$. The $weight$ is continuous with respect to the weak* topology because $B_{e_i}(k+1)\in\Sigma$, thus $weight(\mu)=0$ as required.
\end{proof}

\begin{remk}
When we apply this theorem later in the paper it would be enough for $\mu$ to only be defined on the algebra $\{C_{k+1}(\Gamma)\cap A|\,A\in\Sigma\}$, where $\Sigma$ is the algebra generated by the $A_{v,j}$, rather than having $\mu$ defined on all measurable subsets of $C_{k+1}(\Gamma)$. If we had modified the theorem to only require this, then the last part of the proof that argues about convergence would be easier because we wouldn't need weak* convergence or Caratheodory's Extension theorem (the reason is basically that $\Sigma$ is countable). However keeping the theorem as it is makes for a cleaner statement.
\end{remk}

\section{Colouring walls}\label{sec:colouringwalls}

In this section we define the local colouring data that will play a key role in the following sections. This local colouring data will be defined as equivalence classes on the space of all colourings of walls in $\mathcal{X}$; first we define this space of colourings by building a graph out of the walls of $\mathcal{X}$. This section follows Definitions 6.6-6.8 of Agol's paper, but with a simplification implied by Agol's ICM notes \cite[7.4]{icm} and with some differences in notation.

\begin{defn}(The graph $\Gamma$)\\\label{gamma}
Let $\Gamma=\Gamma(\mathcal{X})$ be a graph whose vertices are the walls of $\mathcal{X}$, and with walls $\overline{W}_1, \overline{W}_2$ joined by an edge in $\Gamma$ if $d(\overline{W}_1,\overline{W}_2)\leq R$. We have a natural action of $\mathcal{G}$ on $\Gamma$. As $\mathcal{X}$ is locally finite, cocompact and with finite walls (see Lemma \ref{propmathcalX}) it follows that the degree of vertices in $\Gamma$ is bounded by some $k\in\mathbb{N}$.

As in Definition \ref{colourings}, we have an action of $\mathcal{G}$ on $C_{k+1}(\Gamma)$ by $c\mapsto gc:=c\circ g^{-1}$ - this also induces an action of $G$ on $C_{k+1}(\Gamma)$ by $c\mapsto gc:=c\circ \phi(g^{-1})$.
\end{defn}

\begin{defn}(Equivalent colourings)\\\label{equivcolourings}
For $W$ a wall in $X$, we will define equivalence classes $[-]_W$ in $C_{k+1}(\Gamma)$ that depend only on the colour of vertices `near' to $\overline{W}$ in $\Gamma$. What we mean by vertices `near' to $\overline{W}$ will depend on the colouring $c\in C_{k+1}(\Gamma)$ in question. Specifically, define the equivalence class $[c]_W$ by 
\begin{equation*}
[c]_W:=\{c'\in C_{k+1}(\Gamma)\,|\,c'=c\text{ on the ball of radius $c(\overline{W})$ in $\Gamma$ centred at $\overline{W}$}\}.
\end{equation*}
For $e\in E(X)$ that crosses a wall $W$, we will use $[-]_e$ as an alternative notation for the equivalence class $[-]_W$. We will also use the notation $c(e):=c(\overline{W})$. Note that $c'\in[c]_e$ implies $c'(e)=c(e)$. For $x\in V(X)$ we will require the finer equivalence classes $[-]_x$ in $C_{k+1}(\Gamma)$ defined by 
\begin{equation*}
[c]_x:=\cap\{[c]_e\,|\,e\in E(X)\text{ incident at }x\}.
\end{equation*} 
We want to combine these families of equivalence relations into two $G$-invariant equivalence relations, one for edges and one for vertices. Do this as follows. Define equivalence classes on $E(X)\times C_{k+1}(\Gamma)$ by $[e,c]:=\{e\}\times [c]_e$. Similarly, define equivalence classes on $V(X)\times C_{k+1}(\Gamma)$ by $[x,c]:=\{x\}\times[c]_x$. The action of $G$ on $E(X), V(X)$ and $C_{k+1}(\Gamma)$ induces actions on the product spaces and on the two sets of equivalence classes by $g[e,c]=[ge,gc]$ and $g[x,c]=[gx,gc]$. It is straightforward to check that these are well-defined (first check that $g[c]_W=[gc]_{gW}$).
\end{defn}

\begin{remk}\label{finiteequiv}
For each $e\in E(X)$, the classes $[-]_e$ only depend on the colour of vertices in some $(k+1)$-ball of $\Gamma$, and so there are only finitely many of these equivalence classes. Similarly, for each $v\in V(X)$, there are only finitely many classes $[-]_v$. As there are finitely many $G$-orbits in  $E(X)$ and $V(X)$, there must only be finitely many $G$-orbits of equivalence classes on $E(X)\times C_{k+1}(\Gamma)$ and $V(X)\times C_{k+1}(\Gamma)$.
\end{remk}

\begin{remk}\label{nbhedges}
	If edges $e,f\in E(X)$ are both incident at a vertex $x$, then $d(\overline{W(e)},\overline{W(f)})\leq1<R$, so $c(e)=c(\overline{W(e)})\neq c(\overline{W(f)})=c(f)$ for any $c\in C_{k+1}(\Gamma)$.
\end{remk}

\section{Starting the gluing construction}\label{sec:startgluing}

In this section we introduce the main construction used to prove Theorem \ref{mainthm}. We will implement the first step of the construction and show how the theorem follows from the final step. The process of going between steps will be left to the last two sections. The construction is similar in spirit to \S8 of Agol's paper, but we work with subspaces of $X$ rather than orbi-complexes; in other words we work on the space where $G$ acts rather than in the quotient. This will allow us to make use of the CAT(0) geometry of $X$, as we do in sections 8 and 9, and it will save us from having to recall technical background about orbi-complexes. Lemma \ref{Vk+1} is based on the ideas of \S7 of Agol's paper.

To prove Theorem \ref{mainthm} we will inductively construct $\mathcal{V}_{k+1}, \mathcal{V}_k,..., \mathcal{V}_0$ (with the $k$ from Definition \ref{gamma}). Each $\mathcal{V}_j$ will be a non-empty collection of triples $(Z,H,(c_x))$ where $Z\subset X$ is a non-empty intersection of half-spaces (so is closed and convex), and for each $x\in Z$ a vertex we have $c_x\in C_{k+1}(\Gamma)$ a colouring.  $H$ will be a subgroup of $G$ that acts on $Z$ freely and cocompactly, and $c_{hx}=hc_x$ for $h\in H$. We will permit $\mathcal{V}_j$ to contain duplicates of some triples. Where there is no danger of ambiguity, we will write $Z\in\mathcal{V}_j$ as shorthand for $(Z,H,(c_x))\in\mathcal{V}_j$.

$\mathcal{V}_j$ will satisfy four properties; before stating these formally we give some loose motivation for them. We will often work with the finite complex $Z/H$ which has universal cover $Z$ (as $Z$ is CAT(0)). Technically $Z$ and $Z/H$ are not quite cube complexes, rather they are something like `cube complexes with boundary walls' - we'll just have to live with this, but we will get a genuine cube complex once we've finished all the gluing. Think of the colourings $c_x$ as giving information about some of the walls nearby $x$, such as which walls mark the boundary of $Z$. The rough idea of the construction is to glue the complexes $Z/H$ along walls coloured $j$ to form $\mathcal{V}_{j-1}$. Neighbouring vertices in $Z$ must agree about colours of nearby walls, and so vertices next to boundary walls that will later be glued up must have a potential matching in which colourings are compatible.

Here are the properties that $\mathcal{V}_j$ must satisfy (which use notation from Definitions \ref{equivcolourings} and \ref{QVH}):

\begin{enumerate}[(1)]
\item If $e\in E(X)$ joins vertices $x,y\in Z\in\mathcal{V}_j$, then $[e,c_x]=[e,c_y]$. So adjacent vertices are equipped with similar colourings.
\item If $e\in E(X)$ joins vertices $x, y$ with $x\in Z\in\mathcal{V}_j$, then
\begin{equation*}
y\in Z \,\Leftrightarrow\, c_x(e)>j.
\end{equation*}
So walls in the interior of $Z$ are coloured $>j$ by their neighbouring vertices, whereas walls on the boundary of $Z$ are coloured $\leq j$.
\item (Gluing Equations)

Let $f\in E(X)$ and $c\in C_{k+1}(\Gamma)$. Define sets
\begin{equation*}
\mathcal{V}_j^{\pm}(f,c):=\{(H\cdot e, Z)\,|\,(Z,H,(c_x))\in\mathcal{V}_j,\,\exists g\in G:\, gZ\cap W(f)^\pm\neq\emptyset,\,g[e,c_e]=[f,c]\},
\end{equation*}
where any duplicates of triples $(Z,H,(c_x))\in \mathcal{V}_j$ are counted separately. In other words, $\mathcal{V}_j^{\pm}(f,c)$ is the set of edges in the complexes $Z/H$ that, modulo the action of $G$, correspond to $f$ and have colouring in the class $[c]_f$; and the $\pm$ means that $Z/H$ continues on the $\pm$ side of the wall dual to the edge. 
 
The Gluing Equations are given by
\begin{equation*}
|\mathcal{V}_j^+(f,c)|=|\mathcal{V}_j^-(f,c)|,
\end{equation*}  
where $f$ ranges over $E(X)$ and $c$ ranges over $C_{k+1}(\Gamma)$. Roughly speaking, these equations will ensure that walls on the boundary of complexes $Z/H$ that look like $W(f)$ on the $W(f)^+$ side can be matched up with those that look like $W(f)$ on the $W(f)^-$ side, with compatible colourings matched together - but there is more work to be done later to arrange this precisely.
\item $H\in\mathcal{QVH}$ for any triple $(Z,H,(c_x))\in\mathcal{V}_j$. This will allow us to make use of the theorems in Section \ref{sec:special}.
\end{enumerate}
It will be convenient to also define colourings $c_e\in C_{k+1}(\Gamma)$ for edges $e\in E(X)$ that intersect $Z$, such that if $e$ is incident at a vertex $x\in Z$ then $[e,c_e]=[e,c_x]$. This is possible by property (1). We will only ever care about the class $[e,c_e]$, so the colourings $(c_e)$ are not extra data that needs to be added to the triple $(Z,H,(c_x))$.

To start the inductive construction of the $\mathcal{V}_j$, we first define $\mathcal{V}_{k+1}$.

\begin{lem}\label{Vk+1}
	There exists $\mathcal{V}_{k+1}$ satisfying all of the above conditions.
\end{lem} 
\begin{proof}
Let $\{x_1,...,x_t\}$ be a complete set of $G$-orbit representatives in $V(X)$. For each $i$ pick colourings $c_{il}\in C_{k+1}(\Gamma)$ so that we have a partition
\begin{equation}\label{xjpartition}
C_{k+1}(\Gamma)=\bigsqcup_{1\leq l\leq n_i}[c_{il}]_{x_i}.
\end{equation} 
These partitions are finite by Remark \ref{finiteequiv}. For each $(i,l)$ let $Z_{il}$ be the intersection of all half-spaces containing $x_i$. Note that $Z_{il}$ will be compact, in fact it will be the union of cubes in $\dot{X}$ (the barycentric subdivision of $X$) that contain $x_i$, so $Z_{il}\cap V(X)=\{x_i\}$. We will then define $\mathcal{V}_{k+1}$ to be the collection of triples $(Z_{il},\{1\},c_{il})$.
 
We must check that this definition of $\mathcal{V}_{k+1}$ satisfies properties (1)-(4) above. Each $Z_{il}$ only contains one vertex of $X$, so properties (1) and (2) hold vacuously, and (4) is also immediate since the trivial group is in $\mathcal{QVH}$. However (3) might not hold. To rectify this we will make $\mathcal{V}_{k+1}$ contain $\alpha_{il}$ copies of $Z_{il}$ for appropriate integers $\alpha_{il}$, which we will spend the rest of the proof constructing.
 
Take $f\in E(X)$ with endpoints $x_+\in W(f)^+$ and $x_-\in W(f)^-$, and take $c\in C_{k+1}(\Gamma)$. Say $x_+$ is in the orbit of $x_i$. How can we count $\mathcal{V}^+_{k+1}(f,c)$? Well the contributions will come from precisely the pairs $(c_{il},e)$, with $e$ incident at $x_i$, such that there exists $g\in G$ with $gx_i=x_+$ and $g[e,c_{il}]=[f,c]$ - and each pair will contribute $\alpha_{il}$. Note that $c_{il}(e)=gc_{il}(f)=c(f)$, so by Remark \ref{nbhedges} there is at most one valid choice of edge $e$ for each choice of colouring $c_{il}$ - call this edge $e_l$. Another way we could try to count $\mathcal{V}^+_{k+1}(f,c)$ is by counting pairs $(g,l)$ with $g\in G$ and $1\leq l\leq n_i$ such that $gx_i=x_+$ and $g[e_l,c_{il}]=[f,c]$, and let each such $(g,l)$ contribute $\alpha_{il}$. This method will over-count, but we can measure the extent of this over-counting by the following claim, where $M^+(f,c)$ denotes the total obtained by this over-counting method.

\begin{claim}
$M^+(f,c)=|\text{Stab}_G([f,c])||\mathcal{V}^+_{k+1}(f,c)|$
\end{claim}

\begin{claimproof}
Fix $l$ such that $(c_{il},e_l)$ contributes to $\mathcal{V}^+_{k+1}(f,c)$. We claim that $(g,l)$ contributes to $M^+(f,c)$ if and only if $g[e_l,c_{il}]=[f,c]$. The `only if' direction is immediate; for the `if' direction we just need to check that $g[e_l,c_{il}]=[f,c]$ implies $gx_i=x_+$. Indeed we assumed that $(c_{il},e_l)$ contributes to $\mathcal{V}^+_{k+1}(f,c)$, so there is some $g\in G$ with $ge_l=f$ and $gx_i=x_+$, but then Remark \ref{noswap} implies that any $g\in G$ with $ge_l=f$ satisfies $gx_i=x_+$. 

Since $G$ acts on the equivalence classes of $E(X)\times C_{k+1}(\Gamma)$, we deduce that there are $|\text{Stab}_G([f,c])|$ elements $g\in G$ such that $(g,l)$ contributes to $M^+(f,c)$. Each $(g,l)$ will contribute $\alpha_{il}$ to $M^+(f,c)$, and $(c_{il},e_l)$ also contributes $\alpha_{il}$ to $|\mathcal{V}^+_{k+1}(f,c)|$, so the claim follows.
\end{claimproof}

Why on earth is it useful to over-count $\mathcal{V}^+_{k+1}(f,c)$? Well the factor of over-counting we obtained only depends on $f$ and $c$, so if we over-count $\mathcal{V}^-_{k+1}(f,c)$ in the same way to produce a total $M^-(f,c)$, then the factor of over-counting is the same. Hence the Gluing Equation $|\mathcal{V}_{k+1}^+(f,c)|=|\mathcal{V}_{k+1}^-(f,c)|$ is equivalent to $M^+(f,c)=M^-(f,c)$. The trick now is to solve these transformed Gluing Equations by using the measure from Theorem \ref{measure}.

The $M^\pm(f,c)$ are just integer sums of the $\alpha_{il}$. We want the $\alpha_{il}$ to be non-negative integers (that are not all zero), but as a start we will exhibit positive real numbers $\alpha_{il}$ that solve the Gluing Equations. We do this by taking the measure $\mu$ from Theorem \ref{measure} applied to the graph $\Gamma$ with the action of $G$, and putting
\begin{equation}\label{alpha}
\alpha_{il}=\frac{\mu([c_{il}]_{x_i})}{|\text{Stab}_G(x_i)|}.
\end{equation}

Next, observe that $[c]_f$ can be partitioned into $[-]_{x_+}$ equivalence classes, which can be written 
\begin{equation}
[c]_f=\bigsqcup_{b\in\beta_+(f,c)}[c_b]_{x_+},
\end{equation}
for some $c_b\in C_{k+1}(\Gamma)$ (and this partition is finite by Remark \ref{finiteequiv}). A pair $(g,l)$ contributes to $M^+(f,c)$ if and only if $g[x_i,c_{il}]=[x_+,c_b]$ for some $b\in\beta_+(f,c)$, so we can count $M^+(f,c)$ by adding up the contributions from each $c_b$. There will be $|$Stab$_G(x_i)|$ choices of $g$ with $gx_i=x_+$, and for each pair $(g,b)$ there will be a unique $l$ with $g[x_i,c_{il}]=[x_+,c_b]$. As $\mu$ is $G$-invariant, we see from (\ref{alpha}) that each $\alpha_{il}$ only depends on the $G$-orbit of $[x_i,c_{il}]$, and so the contribution to $M^+(f,c)$ from a given $c_b$ will equal $\alpha_{il}|$Stab$_G(x_i)|$ for any $l$ with $[x_+,c_b]\in G\cdot[x_i,c_{il}]$. For any such $l$, the $G$-invariance of $\mu$ implies that $\mu([c_b]_{x_+})=\mu([c_{il}]_{x_i})$, hence
\begin{align*}
M^+(f,c)&=\sum_{b\in\beta_+(f,c)}\frac{\mu([c_b]_{x_+})}{|\text{Stab}_G(x_i)|}|\text{Stab}_G(x_i)|\\
&=\sum_{b\in\beta_+(f,c)}\mu([c_b]_{x_+})\\
&=\mu([c]_f).
\end{align*} 
Again this only depends on $f$ and $c$, so our clever choices of $\alpha_{il}$ will also give $M^-(f,c)=\mu([c]_f)$. This will hold for all $f$ and $c$, thus solving the Gluing Equations.
 
All that remains is to convert this into a non-negative integer solution of the Gluing Equations. Note that, as functions of the $\alpha_{il}$, $M^\pm(f,c)$ only depend on the $G$-orbit of $[f,c]$; there are finitely many such orbits by Remark \ref{finiteequiv}, so there are actually just finitely many Gluing Equations (as equations in the $\alpha_{il}$). Note that the values of $\alpha_{il}$ from (\ref{alpha}) are not all zero, because, for fixed $i$, $C_{k+1}(\Gamma)$ can be expressed as a finite partition of $[-]_{x_i}$ equivalence classes as in (\ref{xjpartition}), and $\mu(C_{k+1}(\Gamma))=1$. We can then promote our non-negative real number solution of the Gluing Equations to a non-negative integer solution using the following claim. Moreover, since our real number solution isn't identically zero, we can arrange that the integer solution isn't identically zero either. This is an example of linear programming, a technique which has been widely used in topology, an early instance being Haken's work on normal surface theory \cite{normalsurf}.

\begin{claim}
Let $A$ be an integer matrix defining a linear map $A:\mathbb{R}^n\to\mathbb{R}^m$. If $\exists v\in\ker{A}-\{0\}$ with non-negative entries, then $\exists w\in\ker{A}-\{0\}$ with non-negative integer entries.
\end{claim}

\begin{claimproof}
Let $v\in\ker{A}-\{0\}$ have non-negative entries. In fact we may assume that all entries of $v$ are strictly positive (else delete columns in $A$ corresponding to the zero entries of $v$ and solve the claim for this matrix, and reintroduce the zero entries to $w$ afterwards). It suffices to find $w$ with non-negative rational entries since we can multiply out denominators to make the entries integers. Now $A\mathbb{R}^n$ is the closure of $A\mathbb{Q}^n$ so both have the same dimension as vector spaces over $\mathbb{R}$ and $\mathbb{Q}$ respectively; thus $\ker{(A)}$ and $\ker{(A)}\cap\mathbb{Q}^n$ also have the same dimensions, and so the former must be the closure of the latter. Therefore we can choose $w\in\ker{(A)}\cap\mathbb{Q}^n$ to be a rational approximation of $v$, close enough so that it has positive entries.
\end{claimproof}
\end{proof}

The inductive construction of $\mathcal{V}_k,..., \mathcal{V}_0$ will be left to the final two sections. To close this section we show that Theorem \ref{mainthm} follows from the existence of $\mathcal{V}_0$.

\begin{breakproof}[Proof of theorem \ref{mainthm}, given $\mathcal{V}_0$]
Take some triple $(Z,H,(c_x))\in\mathcal{V}_0$. Property (2) and the connectedness of $X$ imply that $Z=X$. $H$ acts cocompactly on $X$ so must be finite index in $G$. $H$ acts freely on $X$ by definition of $\mathcal{V}_0$. Property (4) in conjunction with Theorem \ref{qvhspecial} and Corollary \ref{bothspecial} tells us that $X/H$ is virtually special. We can then take $G'<H$ finite index such that $X/G'$ is special.
 \end{breakproof}

\section{Controlling boundary walls}\label{sec:controlling}

To go from $\mathcal{V}_j$ to $\mathcal{V}_{j-1}$ we will glue together the various complexes $Z/H$ along the quotients of certain `boundary walls'. In this section we will establish what boundary walls are and which ones we are gluing along, and we will prove some technical lemmas (to be used later) that control the behaviour of these walls. Lemma \ref{jbwallcolouring} comes from p1062 of Agol's paper, and Lemma \ref{acylindrical} comes from p1063, but both are recast to fit with our definitions. A novel feature of our argument is the notion of `portals', which we introduce in Definition \ref{jbandportal}; these are the parts of boundary walls that we want to glue along - they will be used extensively in the final section.

For this section fix $(Z,H,(c_x))\in\mathcal{V}_j$.

\begin{defn}(Boundary walls)\\
For $e$ an edge crossing out of $Z$ we call $W(e)$ a \textit{boundary wall} of $Z$. Equivalently, boundary walls are walls $W(e)$ for $e$ an edge intersecting $Z$ and $c_e(e)\leq j$. $Z$ is an intersection of half-spaces, so if $W$ is a boundary wall then $Z$ is contained in one half-space of $W$. Let $\partial Z\subset Z$ be the union of all boundary walls intersected with $Z$.
\end{defn}

\begin{remk}\label{Zcon}
For vertices $x,y\in V(X)$ with $x\in Z$, let $\gamma$ be a shortest edge path from $x$ to $y$, then the following are equivalent,
\begin{enumerate}[(1)]
\item $y\notin Z$,
\item $\gamma$ crosses a boundary wall,
\item $y$ and $z$ are separated by a boundary wall.
\end{enumerate}
Indeed if $y\notin Z$ then the first time $\gamma$ leaves $Z$ it must cross a boundary wall, so (1) implies (2). (2) implies (3) follows from Proposition \ref{edgegeodesic}. (3) implies (1) because $Z$ is contained in one half-space of the boundary wall. In particular this shows that any two vertices in $Z$ are connected by an edge path that stays in $Z$.
\end{remk}

\begin{lem}\label{intersectbwall}
Let $W_1,...,W_n$ be pairwise intersecting walls, with $Z\cap W_i\neq\emptyset$ for each $i$, then $Z$ contains a vertex $x$ incident at edges $e_1,..., e_n$ that are dual to $W_1,...,W_n$ respectively (and so $e_1,..., e_n$ form the corner of an $n$-cube in $X$).
\end{lem}
\begin{proof}
It suffices to prove the lemma for $W_1,...,W_n$ a maximal family of pairwise intersecting walls with $Z\cap W_i\neq\emptyset$ for each $i$. Let $y\in Z$ be a vertex. By Helly's Theorem for cube complexes (Proposition \ref{intersectwalls}) there is an $n$-cube $C$ intersected by the walls $W_1,...,W_n$, and $C$ has a vertex $x$ such that no $W_i$ separates $x$ and $y$. Consider a shortest edge path from $y$ to $x$; if it crosses no boundary walls then $x\in Z$ and we are done. Suppose it does cross a boundary wall, $W$ say. By Proposition \ref{edgegeodesic}, $W$ divides $X$ into two half-spaces, one containing $x$ and the other containing $y$ and $Z$. For each $i$, part of $W_i$ is in the half-space containing $x$, but $Z\cap W_i\neq\emptyset$, so we must also have $W\cap W_i\neq\emptyset$, contradicting the maximality of $W_1,...,W_n$.
\end{proof}

If there are two edges crossing different boundary walls (possibly in different triples of $\mathcal{V}_j$), and these edges give the same colouring equivalence class, then we want to be able to `zip' together these boundary walls in a colour-compatible way. The following lemma will help us to achieve this, although we won't actually do the zipping until Section 9.

\begin{lem}\label{jbwallcolouring}(Zipping Lemma)\\
Let $W$ be a boundary wall of $Z$. Then the edges $e$ crossing out of $Z$ with $W=W(e)$ all induce the same class $[c_e]_W$ and hence give the same colour $c_e(\overline{W})=c_e(e)$ (this can be thought of as the colour of $W$, and we'll refer to it as such).
\end{lem}
\begin{proof}\renewcommand{\qedsymbol}{}
\end{proof}
\begin{wrapfigure}{r}{.5\textwidth}
\centering
\includegraphics[width=0.4\textwidth, clip=true ]{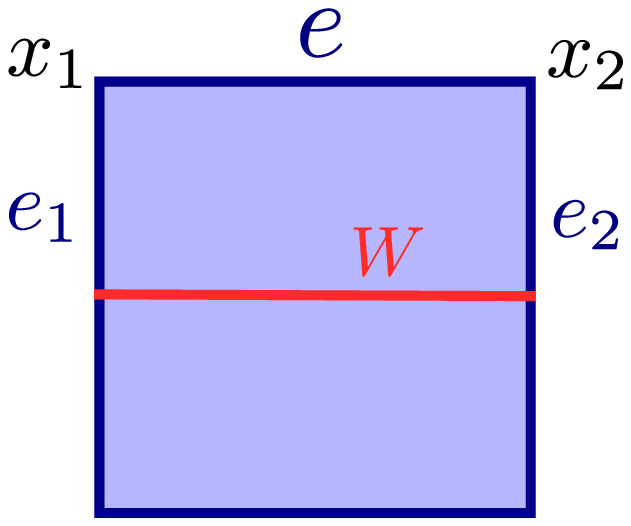}
\end{wrapfigure}
Let $S$ be a square in $X$ with an edge $e$ joining vertices $x_1,x_2\in Z$, let $e_1, e_2$ be the other edges incident at $x_1, x_2$ respectively and suppose they cross out of $Z$ with $W(e_1)=W(e_2)=W$. By property (2) of $\mathcal{V}_j$, $c_{x_1}(e)>j\geq c_{x_1}(e_1)$. $\overline{W(e)}$ and $\overline{W}$ intersect, so are adjacent vertices in $\Gamma$. Now $[c_{x_1}]_{W(e)}=[c_e]_{W(e)}$, so $c_{x_1}$ agrees with $c_e$ on the ball of radius $c_{x_1}(e)$ about $\overline{W(e)}$ in $\Gamma$. But this ball contains the ball of radius $c_{x_1}(e_1)$ about $\overline{W}$, hence $[c_{x_1}]_W=[c_e]_W$. Similarly $[c_{x_2}]_W=[c_e]_W$. So $[c_{e_1}]_W=[c_{x_1}]_W=[c_{x_2}]_W=[c_{e_2}]_W$.

$Z$ is an intersection of half-spaces, so $W\cap Z$ is an intersection of half-spaces in the induced cube structure on $W$; so by Remark \ref{Zcon} (applied to $W\cap Z\subset W$ instead of $Z\subset X$) any two vertices in $W$ that lie in $Z$ are joined by an edge path in $W$ that stays in $Z$. Vertices in $W$ that lie in $Z$ correspond to edges dual to $W$ that cross out of $Z$, and an edge in $W$ that lies in $Z$ corresponds to a square, as above, joining edges $e_1,e_2$ dual to $W$ that cross out of $Z$. Thus the lemma follows from the fact that $[c_{e_1}]_W=[c_{e_2}]_W$. \qed

\begin{defn}($j$-boundary walls and portals)\label{jbandportal}\\
	
	\begin{minipage}{.6\textwidth}
		\includegraphics[width=.9\textwidth]{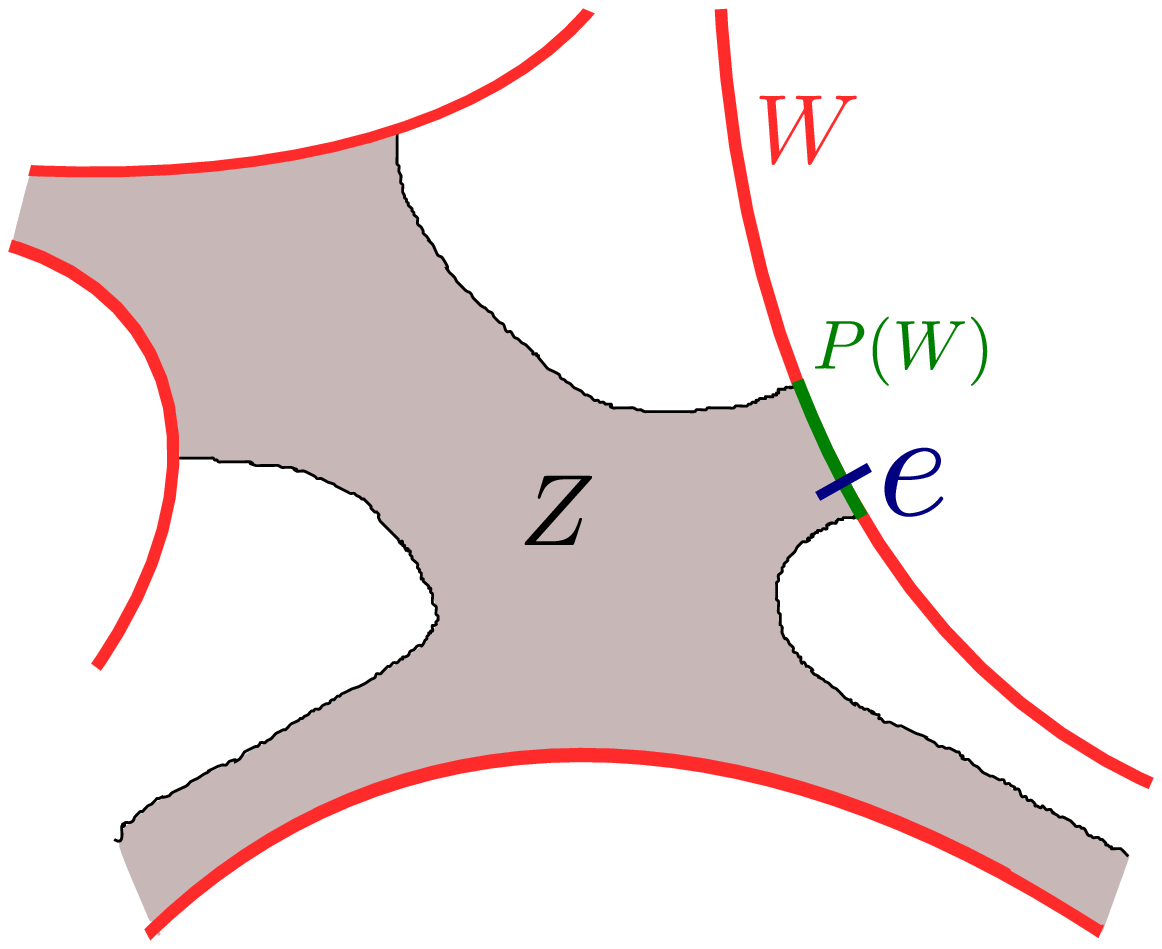}
	\end{minipage}
\begin{minipage}{.4\textwidth}
If a boundary wall of $Z\in\mathcal{V}_j$ has colour $j$ (in the sense of the Zipping Lemma), call it a \textit{$j$-boundary wall}. For $W$ a $j$-boundary wall, let $P(W):=Z\cap W$ be the \textit{portal of $W$ leading to $Z$}. Whenever we talk about a portal $P$ it will implicitly be equipped with a choice of $Z\in\mathcal{V}_j$ that it leads to. If an edge $e$ dual to $W$ crosses out of $Z$, say that $e$ is \textit{dual} to $P(W)$. These are shown in the picture to the left, with $j$-boundary walls in red and other boundary walls in black. Note that portals need not be bounded. Let $\partial_j Z\subset\partial Z$ denote the union of all portals leading to $Z$.
\end{minipage}

\end{defn}

As alluded to at the beginning of this section, we will be gluing together the various $Z/H$ along the $H$-quotients of their portals. To facilitate this we now establish some lemmas that control the behaviour of $j$-boundary walls and portals.

\begin{lem}\label{disjointboundary}
A vertex in $Z$ cannot be incident at distinct edges dual to $j$-boundary walls. Moreover, any two $j$-boundary walls are disjoint.
\end{lem}
\begin{proof}
Suppose	there is a vertex $x\in Z$ incident at distinct edges dual to $j$-boundary walls $W_1$ and $W_2$. By Proposition \ref{edgewall}, $W_1\neq W_2$. Furthermore, we know from Lemma \ref{propmathcalX}(3) that $W_1$ and $W_2$ map to distinct walls $\overline{W}_1$ and $\overline{W}_2$ in $\mathcal{X}$. Since $W_1$ and $W_2$ are $j$-boundary walls, we have that $c_x(\overline{W}_1)=c_x(\overline{W}_2)=j$. But $d(\overline{W}_1,\overline{W}_2)\leq d(W_1,W_2)\leq 1<R$, contradicting $c_x$ being a colouring in $C_{k+1}(\Gamma)$. For the second part of the lemma, if we have two intersecting $j$-boundary walls then we can apply Lemma \ref{intersectbwall} to reduce to the first part of the lemma.
\end{proof}

Gluing together the $Z/H$ along $H$-quotients of portals will form a graph of spaces, and in the corresponding graph of groups the edge groups incident to each vertex group will form a malnormal family by the following lemma.

\begin{lem}\label{acylindrical}
Stabilisers of distinct portals intersect trivially.
\end{lem}
\begin{proof}
Suppose $W_0$ and $W_1$ are distinct walls giving rise to distinct portals $P_0:=Z\cap W_0$ and $P_1:=Z\cap W_1$, and suppose $1\neq h\in H$ stabilises both of them. Since $h$ has no fixed points in $Z$, and since $P_0,P_1\subset Z$ are convex, we deduce that $h$ restricts to hyperbolic isometries of $P_0$ and $P_1$ with translation axes $a_1,a_2$ respectively (see \cite[II.6.8(1) and II.6.10(2)]{nonpos}). Any two translation axes of $h$ are asymptotic, hence we can apply \cite[II.2.13]{nonpos} to see that $a_1$ and $a_2$ bound a flat strip, and by $\delta$-hyperbolicity this strip can have width at most $\delta$. Any point $p$ on $P_0$ is contained in a cube $C$ of $X$; and one of the edges of $C$ closest to $p$ will be dual to $P_0$ and have endpoint $x$ in $Z$, with $d(p,x)\leq\tfrac{1}{2}\sqrt{\dim{C}}\leq\tfrac{1}{2}\sqrt{\dim{X}}$. The same is true for $P_1$, and so there is a path $\beta$ in $Z$ between vertices $x_0,x_1\in Z$ of length at most $\delta+\sqrt{\dim{X}}$ with $x_0,x_1$ being incident at edges dual to $P_0,P_1$ respectively. By considering the sequence of cubes that $\beta$ travels through, there is an edge path $\gamma$ in $Z$ from $x_0$ to $x_1$ with $\gamma\subset N_{\sqrt{\dim{X}}}(\beta)$. Let $\gamma$ have edges $e_1,...,e_n$ and vertices $x_0=y_0,y_1,...,y_n=x_1$. Since $R\geq \delta+2\sqrt{\dim{X}}$ (and thus the mystery of $R$ is revealed!), we have that $d(\overline{W(e_i)},\overline{W_0})\leq d(W(e_i),P_0)\leq R$ for $1\leq i\leq n$, and so $\overline{W(e_i)}$ and $\overline{W_0}$ are adjacent vertices in $\Gamma$. 

For $1\leq i\leq n$, we know from property (1) of $\mathcal{V}_j$ that $[c_{y_{i-1}}]_{e_i}=[c_{y_i}]_{e_i}$, so  $c_{y_i}(\overline{W_0})=c_{y_{i-1}}(\overline{W_0})$. We deduce that $c_{x_1}(\overline{W_0})=c_{x_0}(\overline{W_0})$. And $c_{x_0}(\overline{W_0})=j=c_{x_1}(\overline{W_1})$ since $W_0$ and $W_1$ are $j$-boundary walls. But $d(\overline{W_1},\overline{W_0})\leq d(W_1,W_0)\leq\delta\leq R$, and by \ref{propmathcalX}(3) we know that $\overline{W_0}\neq\overline{W_1}$, hence $\overline{W_0}$ and $\overline{W_1}$ are adjacent vertices in $\Gamma$ which are given the same colour by $c_{x_1}$, a contradiction.
\end{proof}

\begin{defn}(Splitting along colourings)\\\label{coloursplit}
For $\overline{W}$ a wall in $\mathcal{X}$ and $c\in C_{k+1}(\Gamma)$, let $B(\overline{W},c):=\overline{W}\cap\bigcup c^{-1}([1,j])$ be the intersection of $\overline{W}$ with other walls in $\mathcal{X}$ that are coloured $\leq j$ by $c$. Define \textit{$\overline{W}$ split along $c$} by $\overline{W}- c:=\overline{W}- B(\overline{W},c)$ (this will of course depend on $j$, but $j$ is fixed for the rest of the notes so we don't include it in the notation). Working in the barycentric subdivision of $\overline{W}$, $\overline{W}- c$ will be a cube complex with some missing faces corresponding to where we have removed $B(\overline{W},c)$. In general $\overline{W}- c$ will be disconnected, so for a vertex $\bar{x}$ in $\overline{W}$, let $(\overline{W}-c)(\bar{x})$ denote the component of $\overline{W}-c$ containing $\bar{x}$.
\end{defn}

\includegraphics[width=0.55\textwidth, clip=true ]{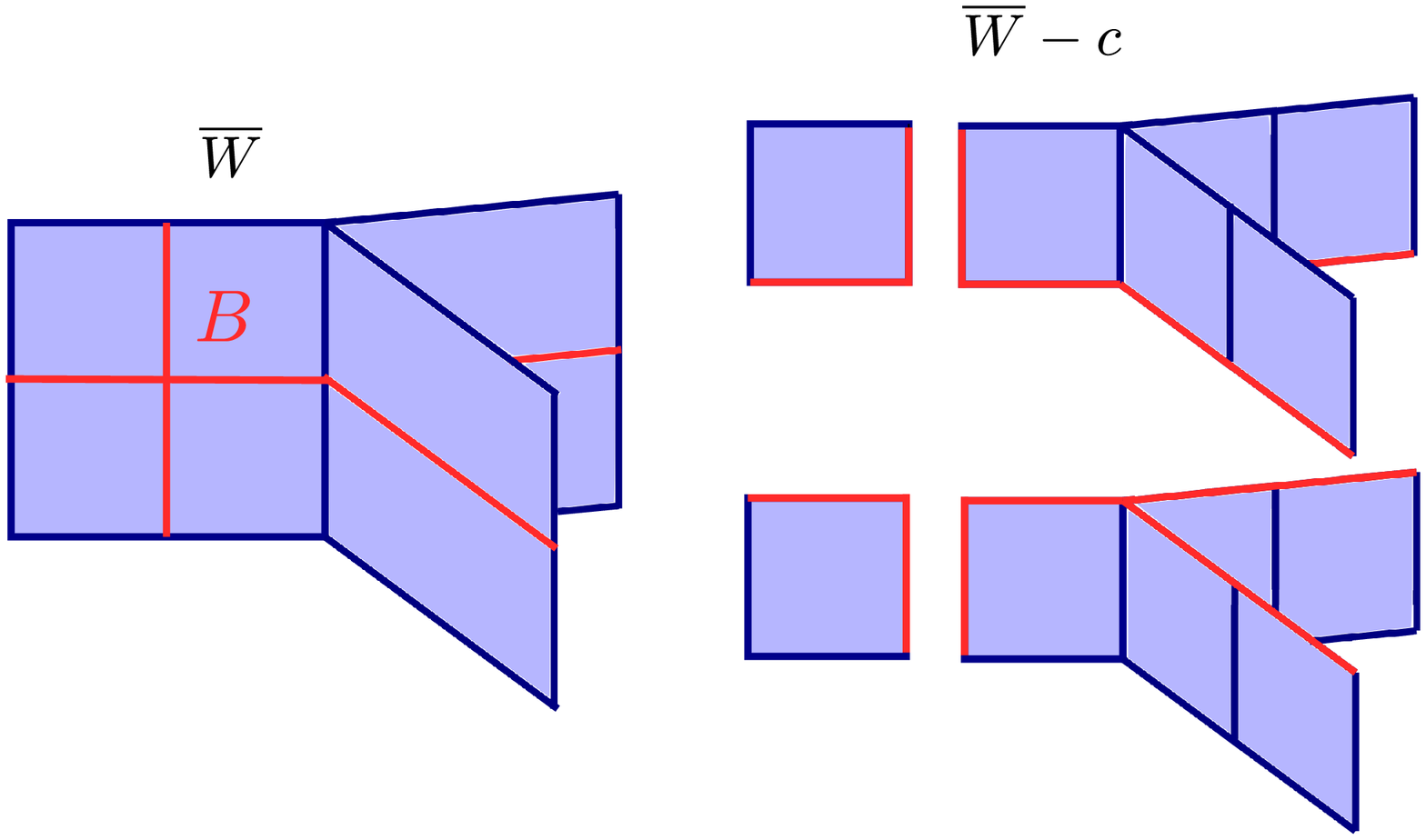}
\bigskip

\begin{lem}(Portal covers)\\\label{partialjcovers}
Let $W$ be a $j$-boundary wall with portal $P=Z\cap W$ and let $e$ be an edge dual to $P$. Let $x_0$ denote the midpoint of $e$ - so $x_0$ is a vertex of $W$.  Then the quotient map $m:X\to\mathcal{X}$ restricts to a universal covering map
\begin{equation}\label{partialjcover}
m|_{\mathring{P}}: \mathring{P}\to(\overline{W}- c_e)(\bar{x}_0),
\end{equation}
where $\mathring{P}$ is the interior of $P$ with respect to the metric topology of $W$ (equivalently $\mathring{P}$ is $P$ minus all boundary walls $W'\neq W$ that intersect $W$). Moreover, $(\overline{W}- c_e)(\bar{x}_0)=(\overline{W}- c)(\bar{x})$ for any other $x\in P$ a vertex of $W$ and any $c\in[c_e]_W$ - in particular, by the Zipping Lemma, $(\overline{W}- c_e)(\bar{x}_0)$ is independent of the choice of $e$ dual to $P$. Furthermore, the group of deck transformations of $m|_{\mathring{P}}$ is $K_P:=\{g\in K\,|\,gx_0\in P\}$ (where $K$ is from Lemma \ref{mathcalX}). 
\end{lem}
\begin{proof}
Consider a wall $W_1$ with $W_1\cap W\cap Z\neq\emptyset$. By Lemma \ref{intersectbwall}, there is a vertex $x\in Z$ with edges $e_1,e_2$ dual to $W_1, W$ respectively. Then
\begin{equation}\label{portboundwall}
c_{e_1}(\overline{W_1})=c_x(\overline{W_1})=c_{e_2}(\overline{W_1})=c_e(\overline{W_1}),
\end{equation}
with the third equality due to the Zipping Lemma. By property (2) of $\mathcal{V}_j$, $W_1$ is a boundary wall if and only if $c_{e_1}(\overline{W_1})\leq j$. So (\ref{portboundwall}) implies that $W_1$ is a boundary wall if and only if $c_e(\overline{W_1})\leq j$.

 The restriction of the quotient map $m:X\to\mathcal{X}=X/K$ certainly defines a map $m|_{\mathring{P}}:\mathring{P}\to\overline{W}\subset\mathcal{X}$ with $m(x_0)=\bar{x}_0$. Now a path $\gamma$ in $P$ based at $x_0$ can go anywhere in $W$ except cross over a boundary wall $W_1$, which by the above arguments is equivalent to $\gamma$ not crossing a wall $W_1$ with $c_e(\overline{W_1})\leq j$, which in turn is equivalent to $m\circ\gamma$ not crossing $B(\overline{W},c_e)$ (see Definition \ref{coloursplit}). This establishes that $m|_{\mathring{P}}$ is a covering. Note that $\mathring{P}$ is equal to $W$ intersected with open half-spaces corresponding to other boundary walls, so it is convex in $X$ and hence simply connected, making $m|_{\mathring{P}}$ a universal covering.
  
 If $c\in[c_e]_W$, then it follows from the definition of $[-]_W$ that, among the walls intersecting $\overline{W}$, the walls coloured $\leq j$ by $c$ are exactly the same as those coloured $\leq j$ by $c_e$ - hence $\overline{W}-c=\overline{W}-c_e$. If we also replace $x_0$ by a different $x\in P$ then clearly $\bar{x}=m_P(x)\in\overline{W}-c$, and so $(\overline{W}-c_e)(x)=(\overline{W}-c)(x)=(\overline{W}-c)(x_0)$.
  
 Finally, we know that $m:X\to\mathcal{X}$ has $K$ as the group of deck transformations, and that $\mathring{P}$ is a component of $m^{-1}((\overline{W}- c_e)(\bar{x}_0))$, and that $K$ acts on these components. So $K_P$ is the stabiliser in $K$ of $\mathring{P}$ (and of $P$), it preserves the covering $m|_{\mathring{P}}$ and it acts transitively on $m|_{\mathring{P}}^{-1}(\bar{x}_0)=P\cap m^{-1}(\bar{x}_0)$ - thus $K_P$ is exactly the group of deck transformations of $m|_{\mathring{P}}$.
\end{proof}

\section{Gluing up walls}\label{sec:gluing}

We are now ready to start constructing $\mathcal{V}_{j-1}$ from $\mathcal{V}_j$. Our strategy will be to glue together different complexes $Z/H$ along the $H$-quotients of portals with `compatible colourings'. To glue together two quotient portals, we need them to be isomorphic as complexes and for this isomorphism to be `colour compatible'; initially it may not be possible to glue up all the portals, but the idea is to make it possible by passing to finite covers of the $Z/H$. The arguments in this section are based on Theorem 3.1 from Agol's paper, for example we quote the same theorems of Haglund-Wise (Theorem \ref{subgroupsep}) and Bestvina-Feighn (Theorem \ref{comb}) - but our arguments contain considerably more detail and will appear quite different as they are recast to work for our set-up. 

\begin{defn}(Compatible portals)\\
We say that two portals $P$ and $P'$ leading to $(Z,H,(c_x)),(Z',H',(c'_x))\in\mathcal{V}_j$ respectively are \textit{compatible} if there are edges $e$ and $f$ dual to $P$ and $P'$ respectively such that $[e,c_e]\in G\cdot[f,c'_f]$.
\end{defn}
 Let $P$ and $P'$ be compatible portals as above, and say they lie in walls $W$ and $W'$. Take $g\in G$ and edges $e$ and $f$ dual to $P$ and $P'$ such that $[e,c_e]=g[f,c'_f]$, and let $x_0$ and $y_0$ be the midpoints of $e$ and $f$. So $e=gf$, $W=gW'$ and $x_0=gy_0$. As $K$ is normal in $G$, we have the following commuting diagram.
\begin{equation}
\begin{tikzcd}[
ar symbol/.style = {draw=none,"#1" description,sloped},
isomorphic/.style = {ar symbol={\cong}},
equals/.style = {ar symbol={=}},
subset/.style = {ar symbol={\subset}}
]
X\ar{r}{g}\arrow{d}{m} &X\arrow{d}{m} \\
\mathcal{X}\arrow{r}{g}&\mathcal{X}\\
\end{tikzcd}
\end{equation}
Then $g$ acts on the wall $\overline{W'}$ to produce
\begin{align*}
g(\overline{W'}-c'_f)(\bar{y}_0)&=(\overline{W}-gc'_f)(\bar{x}_0)\\
&=(\overline{W}-c_e)(\bar{x}_0)&\text{by Lemma \ref{partialjcovers} and the fact that }[c_e]_W=[gc'_f]_W.
\end{align*}
\begin{minipage}{.5\textwidth}
As the maps (\ref{partialjcover}) are coverings for $P$ and $P'$, we deduce that $g$ restricts to a cube isomorphism $\mathring{P'}\to\mathring{P}$ and also $P'\to P$. In fact $P'\to P$ is equivariant with respect to the group isomorphism $K_{P'}\to K_P;\,k\mapsto gkg^{-1}$. This can all be put into the following commutative diagram.
\end{minipage}
\begin{minipage}{.5\textwidth}
\begin{equation}\label{portaltranslate}
\begin{tikzcd}[
ar symbol/.style = {draw=none,"#1" description,sloped},
isomorphic/.style = {ar symbol={\cong}},
equals/.style = {ar symbol={=}},
subset/.style = {ar symbol={\subset}}
]
K_{P'}\arrow{r}{g(-)g^{-1}}[swap]{\sim}&K_P\\
	P'\arrow[loop above, distance=20]\arrow{r}{g}&P\arrow[loop above, distance=20]\\
	\mathring{P'}\arrow[hook]{u}\arrow{d}{m}\arrow{r}{g}&\mathring{P}\arrow[hook]{u}\arrow{d}{m}\\
	(\overline{W'}-c'_f)(\bar{y}_0)\arrow{r}{g}&(\overline{W}-c_e)(\bar{x}_0)
\end{tikzcd}
\end{equation}
\end{minipage}

We also have the following lemma and corollary, which are basically consequences of the Zipping Lemma. These will allow us to group together compatible portals into compatibility classes.

\begin{lem}\label{teleport}(Teleports)\\
Portals $P$ and $P'$ leading to $(Z,H,(c_x)),(Z',H',(c'_x))\in\mathcal{V}_j$ are compatible if and only if there exists $g\in G$ such that
\begin{equation}\label{portcomp}
\{[e,c_e]\,|\,\text{$e$ is dual to $P$}\}=g\{[f,c'_f]\,|\,\text{$f$ is dual to $P'$}\}.
\end{equation} 
In this case we say that $P$ is a $g$\emph{-teleport of $P'$} (note that $P$ could be a $g$-teleport of $P'$ for several different $g$, and that $P'$ could have several different $g$-teleports corresponding to portals that lead to different $Z\in\mathcal{V}_j$).
\end{lem}
\begin{proof}
	Suppose $P$ and $P'$ are compatible. Then there exist edges $e$ and $f$ dual to $P$ and $P'$, and $g\in G$, such that $[e,c_e]= g\cdot[f,c'_f]$. If $f_1$ is another edge dual to $P'$, then $e_1:=gf_1$ is dual to $P$ because, as we showed above, $g:P'\to P$ is an isomorphism. Suppose that $P$ and $P'$ lie in walls $W$ and $W'$. We then have 
	\begin{align*}
g[f_1,c'_{f_1}]&=\{gf_1\}\times g[c'_{f_1}]_{W'}\\
&=\{e_1\}\times g[c'_f]_{W'}&\text{by the Zipping Lemma}\\
&=\{e_1\}\times [gc'_f]_{W}\\
&=\{e_1\}\times [c_e]_{W}&\text{since }[e,c_e]= g\cdot[f,c'_f]\\
&=\{e_1\}\times [c_{e_1}]_{W}&\text{by the Zipping Lemma}\\
&=[e_1,c_{e_1}].
	\end{align*}
	This gives the $\supset$ inclusion in (\ref{portcomp}), and the $\subset$ inclusion follows similarly by considering $g^{-1}:P\to P'$. Conversely, it is immediate that (\ref{portcomp}) implies compatibility of $P$ and $P'$.
\end{proof}
\begin{cor}\label{compatequiv}
	Teleports form a groupoid on the set of portals, meaning that:
	\begin{itemize}
		\item Any portal is a 1-teleport of itself.
		\item If $P$ is a $g$-teleport of $P'$, then $P'$ is a $g^{-1}$-teleport of $P$.
		\item If $P$ is a $g$-teleport of $P'$ and $P'$ is a $g'$-teleport of $P''$, then $P$ is a $gg'$-teleport of $P''$
	\end{itemize}
	In particular, compatibility of portals is an equivalence relation, and we will refer to the equivalence classes as \emph{compatibility classes}. The compatibility class of a portal $P$ is denoted $[P]$.
\end{cor}

For each triple $(Z,H,(c_x))$, $H$ acts on the set of portals leading to $Z$, and these portals are disjoint by Lemma \ref{disjointboundary}. This implies that the map $P/H_P\to Z/H$ is an embedding for each portal $P$. The aim is to glue together quotients $Z/H$ and $Z'/H'$ along portals $P/H_P$ and $P'/H'_{P'}$ for compatible portals $P$ and $P'$. Lemma \ref{teleport} tells us that $P\cong P'$, which is a good start, but it doesn't imply that we get isomorphic quotients $P/H_P\cong P'/H'_{P'}$; the next step is to overcome this by taking finite covers of the complexes $Z/H$, or equivalently by replacing the groups $H$ with finite index subgroups.

\begin{lem}\label{lem:portalisom}
For each triple $(Z,H,(c_x))$ there is a finite index subgroup $\hat{H}\triangleleft H$, such that whenever portals $P$ and $P'$ lead to $(Z,H,(c_x)),(Z',H',(c'_x))\in\mathcal{V}_j$, with $P$ a $g$-teleport of $P'$, then 
\begin{equation}\label{hatHPconjugate}
\hat{H}_P=g\hat{H}'_{P'}g^{-1}.
\end{equation} 
In particular $P/\hat{H}_P\cong P'/\hat{H}'_{P'}$.
\end{lem}
\begin{proof}
To each group $H$ we apply Theorem \ref{thm:command} with respect to the stabilisers of a set of $H$-orbit representatives of portals leading to $Z$. Note that these stabilisers form a malnormal collection by Lemma \ref{acylindrical}. This gives us finite index subgroups $\dot{H}_P< H_P$ for every portal $P$ leading to $Z$, such that $h\dot{H}_P h^{-1}=\dot{H}_{hP}$ for $h\in H$.

Recall from Section \ref{sec:makefinite} the group $K\triangleleft G$ with finite wall quotients. For each portal $P$, define the group
\begin{equation}\label{hatHP}
\hat{H}_P:=K_P\cap\bigcap_{g,P'}g\dot{H}'_{P'}g^{-1},
\end{equation}
where we range over all $g,P'$ such that $P$ is a $g$-teleport of $P'$. This is an intersection of subgroups of $G_P$ that all act cocompactly on $P$ ($K_P$ acts cocompactly by Lemma \ref{partialjcovers}), so they all have finite index in $G_P$.
Moreover, there are finitely many triples in $\mathcal{V}_j$, so there is a uniform bound on the size of the quotients $P'/\dot{H}'_{P'}$, and hence a uniform bound on the indices $|G_P:g\dot{H}'_{P'}g^{-1}|$. This means that $\hat{H}_P$ has finite index in $G_P$. Also note that $1\dot{H}_P 1^{-1}=\dot{H}_P$ is one of the terms in the intersection, so $\hat{H}_P$ is a finite index subgroup of $\dot{H}_P$. Equation (\ref{hatHPconjugate}) follows from (\ref{hatHP}), Corollary \ref{compatequiv} and the normality of $K$ in $G$; in particular this implies $\hat{H}_P\triangleleft H_P$.

Finally, we apply Theorem \ref{thm:command} to obtain finite index subgroups $\hat{H} \triangleleft H$ for each triple in $(Z,H,(c_x))\in\mathcal{V}_j$, such that the $\hat{H}$-stabilisers of portals leading to $Z$ are indeed the subgroups $\hat{H}_P$ defined by (\ref{hatHP}).
\end{proof}

Replacing each $(Z,H,(c_x))$ by $(Z,\hat{H},(c_x))$ would preserve all the properties of $\mathcal{V}_j$ except the Gluing Equations. $(Z,\hat{H},(c_x))$ would contribute $|H:\hat{H}|$ times more to each set $\mathcal{V}_j^\pm (f,c)$ than $(Z,H,(c_x))$ does. But making $|H:\hat{H}|$ copies of $(Z,H,(c_x))$ in $\mathcal{V}_j$ would have the same effect. Therefore, for each triple $(Z,H,(c_x))\in\mathcal{V}_j$, we can replace $(Z,H,(c_x))$ by some number of copies of $(Z,\hat{H},(c_x))$ such that the Gluing Equations are preserved. To simplify notation we will not write $\hat{H}$ for the rest of the section, we will instead assume that the triples in $\mathcal{V}_j$ already satisfy
\begin{equation}\label{HPconjugate}
H_P=gH'_{P'} g^{-1}
\end{equation}
whenever portals $P$ and $P'$ lead to $(Z,H,(c_x)),(Z',H',(c'_x))\in\mathcal{V}_j$, with $P$ a $g$-teleport of $P'$.

\begin{defn}($[P]^+$ and $[P]^-$ -compatible portals)\\
Let $P$ be a portal leading to $Z\in\mathcal{V}_j$, and suppose it lies in a wall $W$. For $P'$ a portal leading to $Z'\in\mathcal{V}_j$ that is compatible with $P$, choose $g\in G$ such that $P$ is a $g$-teleport of $P'$. We say that $P'$ is a $[P]^+$\emph{-compatible portal} if $gZ'\cap W^+\neq\emptyset$ and a $[P]^-$\emph{-compatible portal} if $gZ'\cap W^-\neq\emptyset$ (recall from Proposition \ref{1or2sides} that $W^\pm$ denote the sets of vertices in $N(W)$ on either side of $W$). By Remark \ref{noswap} the labelling $W^\pm$ is $G$-equivariant, so the definition of $[P]^\pm$-compatible portal is independent of the choice of teleport $g$ and of the choice of portal $P$ within the compatibility class.
\end{defn}

If $P$ is a $[P]^+$-compatible portal leading to $(Z,H,(c_x))$ and $P'$ is a $[P]^-$-compatible portal leading to $(Z',H',(c'_x))$, and if $P$ is a $g$-teleport of $P'$, then $Z$ and $gZ'$ lie on opposite sides of $P$ (i.e. $Z\cap gZ'=P$), so we can glue them together to form a larger (convex) subspace of $X$, as illustrated below. Moreover, we can define local colouring data for this subspace by combining $(c_x)$ with the $g$-translate of $(c'_x)$, and these colourings will be compatible on the edges dual to $P$ (in the sense of property (1) of $\mathcal{V}_j$) by Lemma \ref{teleport}.

\begin{figure}[H]
	\definecolor{RED}{HTML}{FF2A2A}
	\definecolor{BLUE}{HTML}{0000FF}
	\centering
	\begin{overpic}[width=.55\textwidth,clip=true]{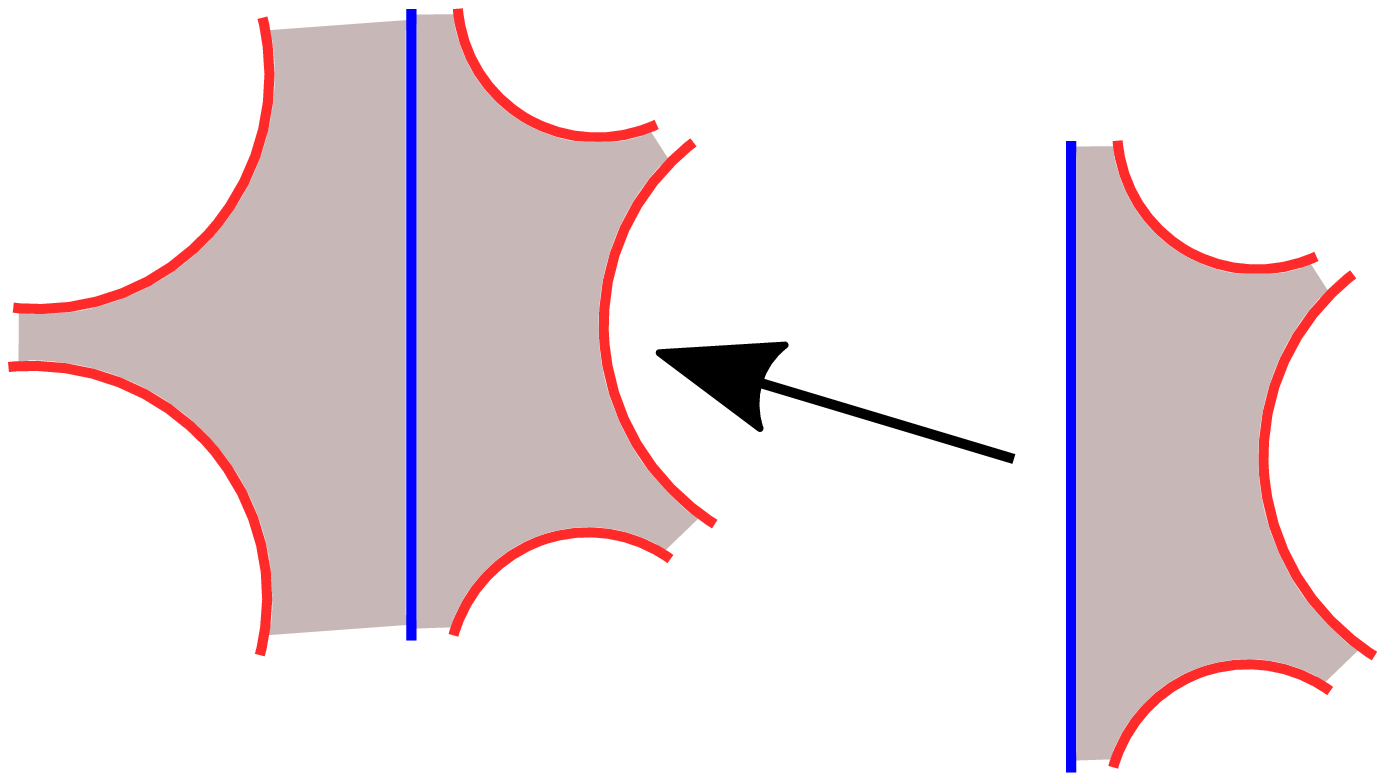}
		\put(32,23){\Huge\textcolor{BLUE}{$P$}}
		\put(68,10){\Huge\textcolor{BLUE}{$P'$}}
		\put(18,30){\Huge$Z$}
		\put(80,20){\Huge$Z'$}
		\put(62,30){\Huge$g$}	
	\end{overpic}
\end{figure}

By (\ref{HPconjugate}) this gluing also descends to the quotients via the following commutative diagram.
\begin{equation}\label{descendglue}
\begin{tikzcd}[
ar symbol/.style = {draw=none,"#1" description,sloped},
isomorphic/.style = {ar symbol={\cong}},
equals/.style = {ar symbol={=}},
subset/.style = {ar symbol={\subset}}
]
P\ar{d}&P'\ar{l}[swap]{g}\ar{d}\\
P/H_P&P'/H'_{P'}\ar{l}[swap]{\bar{g}}
\end{tikzcd}
\end{equation}
Thus $Z/H$ is glued to $Z'/H'$ along $P/H_P\overset{\bar{g}}{\cong} P'/H'_{P'}$. We want to glue up all the portal quotients in pairs, so we need the following definition and lemma.

\begin{defn}(Portal quotients)\\
	For a portal $P$ leading to $(Z,H,(c_x))\in\mathcal{V}_j$ we have a \emph{portal quotient} $P/H_P\xhookrightarrow{} Z/H$ (we consider this as a subspace of $Z/H$, so portals in the same $H$-orbit define the same portal quotient). We define the \emph{size} of $P/H_P$ as
	\begin{align*}
	\text{size}( P/H_P):=|\{H_P\cdot e\,|\,e\in E(X)\text{ is dual to }P\}|.
	\end{align*}
	By (\ref{HPconjugate}) compatible portals have quotients of the same size.
	We say that a portal quotient $P'/H'_{P'}$ is a \emph{$[P]^\pm$-compatible portal quotient} if $P'$ is a $[P]^\pm$-compatible portal.
\end{defn}

\begin{lem}\label{lem:matching}
	For each compatibility class $[P]$, the number of $[P]^+$-compatible portal quotients equals the number of $[P]^-$-compatible portal quotients.
\end{lem}
\begin{proof}
	Fix a portal $P_0$ leading to a triple $(Z_0,H^0,(c^0_x))$, and contained in a wall $W$, and let $[c]_W$ be the associated colouring equivalence class (i.e. $c\in C_{k+1}(\Gamma)$ satisfies $[c]_W=[c^0_f]_W$ for any edge $f$ dual to $P_0$, possible by the Zipping Lemma). The idea is to show that the total size of $[P_0]^+$-compatible portal quotients equals the total size of $[P_0]^-$-compatible portal quotients - the lemma then follows since all portals in the class $[P_0]$ have quotients of the same size. The key is to prove the following equality:
	\begin{align}\label{bigequation}
	\Set{(H\cdot e,Z)\,|\,\begin{aligned}\text{$e\in E(X)$ dual to a $[P_0]^\pm$-compatible}\\
	\text{portal $P$ leading to }(Z,H,(c_x))\in\mathcal{V}_j\end{aligned}}=\bigcup_{f\text{ dual to }P_0}\mathcal{V}_j^\pm(f,c)
	\end{align} 
	\begin{itemize}
		\item Firstly let's check the inclusion $\subset$. Given $(H\cdot e,Z)$ on the LHS of (\ref{bigequation}), if $P_0$ is a $g$-teleport of $P$, then $f:=ge$ is dual to $P_0$ and $[ge,gc_e]=[f,c]$ by Lemma \ref{teleport}. Since $P$ is a $[P_0]^\pm$-compatible portal, we also have $gZ\cap W(f)^\pm\neq\emptyset$, so $(H\cdot e,Z)\in \mathcal{V}_j^\pm(f,c)$ lies on the RHS of (\ref{bigequation}).
		\item For the other inclusion, suppose $(H\cdot e, Z)\in\mathcal{V}_j^\pm(f,c)$ for some edge $f$ dual to $P_0$. Then by definition there is $g\in G$ with $gZ\cap W(f)^\pm\neq\emptyset$ and $g[e,c_e]=[f,c]$. We have $c_e(e)=c(f)=j$, so $e$ is dual to some portal $P$, and $P_0$ is compatible with $P$ by Lemma \ref{teleport}. This implies that $(H\cdot e, Z)$ is on the LHS of (\ref{bigequation}) as required.
	\end{itemize}
Note that $\mathcal{V}_j^\pm(f,c)$ only depends on $\pm$ and the $G$-orbit of the equivalence class $[f,c]$.
It also follows straight from the definition of $\mathcal{V}_j^\pm(f,c)$ that if edges $f_1,f_2$ are dual to $P_0$ with $[f_1,c]\notin G\cdot[f_2,c]$, then $\mathcal{V}_j^\pm(f_1,c)\cap\mathcal{V}_j^\pm(f_2,c)=\emptyset$.
Thus we can write the RHS of (\ref{bigequation}) as a disjoint union by indexing over a finite set of edges $f$ with distinct orbits $G\cdot[f,c]$. Whether we use $+$ or $-$, the Gluing Equations tell us that each set $\mathcal{V}_j^\pm(f,c)$ will be the same size, hence the RHS of (\ref{bigequation}) will also be the same size. In turn, the LHS of (\ref{bigequation}) has the same size for $+$ and $-$, which shows that the total size of $[P_0]^+$-compatible portal quotients equals the total size of $[P_0]^-$-compatible portal quotients, as required.
\end{proof}

\begin{defn}(Graph of spaces $\mathcal{Y}$)\\\label{defn:mathcalY}
For each compatibility class $[P]$ we now choose a matching between the $[P]^+$-compatible portal quotients and $[P]^-$-compatible portal quotients (possible by Lemma \ref{lem:matching}). Define a graph of spaces $\mathcal{Y}$ as follows:
\begin{itemize}
	\item The vertex spaces in $\mathcal{Y}$ are the spaces $Z/H$ for $(Z,H,(c_x))\in\mathcal{V}_j$.
	\item The edge spaces are the portal quotients $P/H_P\xhookrightarrow{} Z/H$ (and we make a choice of portal $P$ from each $H$-orbit of portals leading to $Z$).
	\item For each pair of portal quotients $(P/H_P,P'/H'_{P'})$ from the matching (which we will refer to as an \emph{edge space pair}) we choose $g_P\in G$ such that $P$ is a $g_P$-teleport of $P'$, and glue the portal quotients together using the map $\bar{g}_P:P'/H'_{P'}\to P/H_P$ from (\ref{descendglue}) (and we choose $g_{P'}=g_P^{-1}$).
\end{itemize}
\end{defn}

Now let $Y$ be a connected component of $\mathcal{Y}$. We want to show that $\pi_1(Y)\in\mathcal{QVH}$, but first we must show that it is hyperbolic.
This will rely on (a special case of) the Bestvina-Feighn Combination Theorem \cite[first corollary in \S7]{comb}, albeit stated in slightly more modern language here (see also \cite{Kapovich}).

\begin{thm}(Bestvina-Feighn Combination Theorem)\\\label{comb}
	A graph of groups has hyperbolic fundamental group if it satisfies the following:
	\begin{itemize}
		\item The vertex groups are hyperbolic.
		\item The inclusions of edge groups into vertex groups are quasi-isometric embeddings.
		\item For each vertex group, the collection of incident edge groups is an almost malnormal family (see Definition \ref{defn:almostMalnormal}).
	\end{itemize}
\end{thm}

\begin{lem}\label{Yhyperbolic}
$\pi_1(Y)$ is hyperbolic.
\end{lem}
\begin{proof}
 We just need to check that the conditions of Theorem \ref{comb} are satisfied for the graph of groups corresponding to $Y$.
  
 \begin{itemize}
 	\item The vertex groups come from actions of $H$ on $Z$, which are free cocompact actions on hyperbolic cube complexes.
 	\item The map $P/H_P\xhookrightarrow{} Z/H$ from edge space to vertex space lifts to the inclusion $P\xhookrightarrow{} Z$ of universal covers, which is an isometric embedding, hence the corresponding inclusion of groups $H_P\xhookrightarrow{} H$ will be a quasi-isometric embedding.
 	\item For $(Z,H,(c_x))\in\mathcal{V}_j$, the collection of portal stabilisers $(H_P)$, indexed over some set of $H$-orbit representatives of portals leading to $Z$, is malnormal in $H$ because of Lemma \ref{acylindrical}.
 \end{itemize}
\end{proof}

\begin{lem}\label{YQVH}
$\pi_1(Y)\in\mathcal{QVH}$.
\end{lem}
\begin{proof}
 We have already shown that $\pi_1(Y)$ is hyperbolic; and it has vertex groups $H\in\mathcal{QVH}$ by property (4) of $\mathcal{V}_j$. Finally, the edge groups $H_P$ are quasi-convex in $\pi_1(Y)$ by \cite[Theorem 1.2]{Kapovich}.
\end{proof}

The next step is to embed the universal cover $\tilde{Y}$ of $Y$ into $X$, such that the group of deck transformations of $\tilde{Y}\to Y$ is a subgroup of $G$. We can then complete the inductive step of the gluing construction by promoting $\tilde{Y}$ to a triple in $\mathcal{V}_{j-1}$. The definition of the vertex and edge spaces of $Y$ means that we already have an immersion $Y\to X/G$, so one might think that we can use standard covering space theory to lift this to a map of universal covers $\tilde{Y}\to X$ that is equivariant with respect to a certain subgroup of $G$. Indeed this would follow if the action of $G$ on $X$ were free, but since $G$ might contain torsion the quotient $X/G$ is actually an orbi-complex, which makes the required covering space arguments more subtle. In the following definition and lemma we give an explicit construction of the map $\tilde{Y}\to X$ that does not assume any knowledge of orbi-complex covering space theory. We do this by first constructing $\tilde{Y}$ as a subspace of $X$, and then proving that we have a universal covering $\mu:\tilde{Y}\to Y$.

\begin{defn}\label{tildeTX}($\tilde{Y}$ as a subspace of $X$)\\
Fix $Z_0/H^0$ a vertex space of $Y$ to act as a base vertex space. For $0\leq i\leq n$, suppose $Z_i/H^i$ are vertex spaces of $Y$, suppose $P_i, P'_i$ are portals leading to $Z_i, Z_{i+1}$ respectively such that $(P_i/H^i_{P_i},P'_i/H^{i+1}_{P'_i})$ is an edge space pair, and let $h_i\in H^i$. Then define
\begin{equation}\label{gi}
g_i:=h_0 g_{P_0} h_1 g_{P_1} \cdots g_{P_{i-1}} h_i \in G.
\end{equation}
$\tilde{Y}\subset X$ will be the union of all spaces $g_i Z_i\subset X$ for all choices $Z_i, P_i, h_i$ as above, and the restriction of the covering map $\mu:\tilde{Y}\to Y$ will be given by
 \begin{equation*}
 \mu: g_iZ_i\overset{g_i^{-1}}{\longrightarrow}Z_i \to Z_i/H^i\to Y.
 \end{equation*}
 \end{defn}

\begin{lem}\label{muwelldefined}
$\mu:\tilde{Y}\to Y$ as constructed above is a well-defined universal covering.
\end{lem}
\begin{proof}
Consider $Z_i, P_i, h_i$ for $0\leq i\leq n$ as in Definition \ref{tildeTX}. 
Firstly we will check that the map $\mu$ agrees on the intersection of $g_i Z_i$ and $g_{i+1} Z_{i+1}$. Indeed $g_{P_i}:P'_i\subset Z_{i+1}\to P_i\subset Z_i$ is a teleport between a pair of $[P_i]^+-$ and $[P_i]^--$ compatible portals, so we know that
\begin{align}\label{piecesintersect}
g_i Z_i\cap g_{i+1} Z_{i+1}&=g_i(Z_i\cap g_{P_i} h_{i+1} Z_{i+1})\\\nonumber
&=g_i P_i.
\end{align}
Then the two ways of defining the map $\mu: g_i P_i\to Y$ are given by the following commutative diagram, hence they agree.

\[
\begin{tikzcd}[
ar symbol/.style = {draw=none,"#1" description,sloped},
isomorphic/.style = {ar symbol={\cong}},
equals/.style = {ar symbol={=}},
subset/.style = {ar symbol={\subset}}
]
g_i Z_i & g_i P_i \ar[hook']{l}\ar[hook]{rrr}&&&g_{i+1} Z_{i+1}\\
Z_i \ar{u}{g_i}\ar{d} &P_i \ar{u}{g_i} \ar[hook']{l}\ar{d} &P'_i \ar{l}[swap]{g_{P_i}} \ar[hook]{r} \ar{d} &Z_{i+1} \ar{dr} &Z_{i+1} \ar{d} \ar{l}[swap]{h_{i+1}} \ar{u}{g_{i+1}}\\
Z_i/H^i \ar[bend right=10]{drr}&P_i/H^i_{P_i} \ar[hook']{l} \ar{dr} &P'_i/H^{i+1}_{P'_i} \ar{l}[swap]{\bar{g}_{P_i}} \ar[hook]{rr} \ar{d} && Z_{i+1}/H^{i+1} \ar[bend left=10]{dll}\\
&&Y
\end{tikzcd}
\]

Next we want to investigate the effect of changing $h_i$ to some other element $h'_i\in H$. Suppose this changes $g_i$ to $g'_i$ and $g_{i+1}$ to $g'_{i+1}$.
First note that $g_i Z_i=g'_i Z_i$, and that the map $\mu:g_i Z_i\to Y$ is the same whether defined with $g_i$ or $g'_i$. For $g'_{i+1}$ there are two cases to consider. If $h'_i$ is in the same $H^i_{P_i}$-coset as $h_i$, then (\ref{HPconjugate}) implies that we can write
$$h'_i g_{P_i} h_{i+1}=h_i g_{P_i} h'_{i+1},$$
for some $h'_{i+1}\in H^{i+1}$, so as before we have $g_{i+1} Z_{i+1}=g'_{i+1} Z_{i+1}$ and the map $\mu: g_{i+1} Z_{i+1}\to Y$ is the same either way. On the other hand, if $h'_i$ lies in a different $H^i_{P_i}$-coset than $h_i$, then $g_i P_i$ and $g'_i P_i$ will be distinct portals, and they will lie in disjoint $j$-boundary walls by Lemma \ref{disjointboundary}, hence $g_{i+1} Z_{i+1}$ and $g'_{i+1} Z_{i+1}$ will lie in disjoint half-spaces of $X$ as illustrated below ($j$-boundary wall translates in red). Similarly, if we change $P_i$ to a different $H^i$-orbit of portal leading to $Z_i$, then $g_{i+1} Z_{i+1}$ will be shifted into a third half-space, disjoint from the other two.

\begin{figure}[H]
	\centering
	\begin{overpic}[width=.5\textwidth,clip=true]{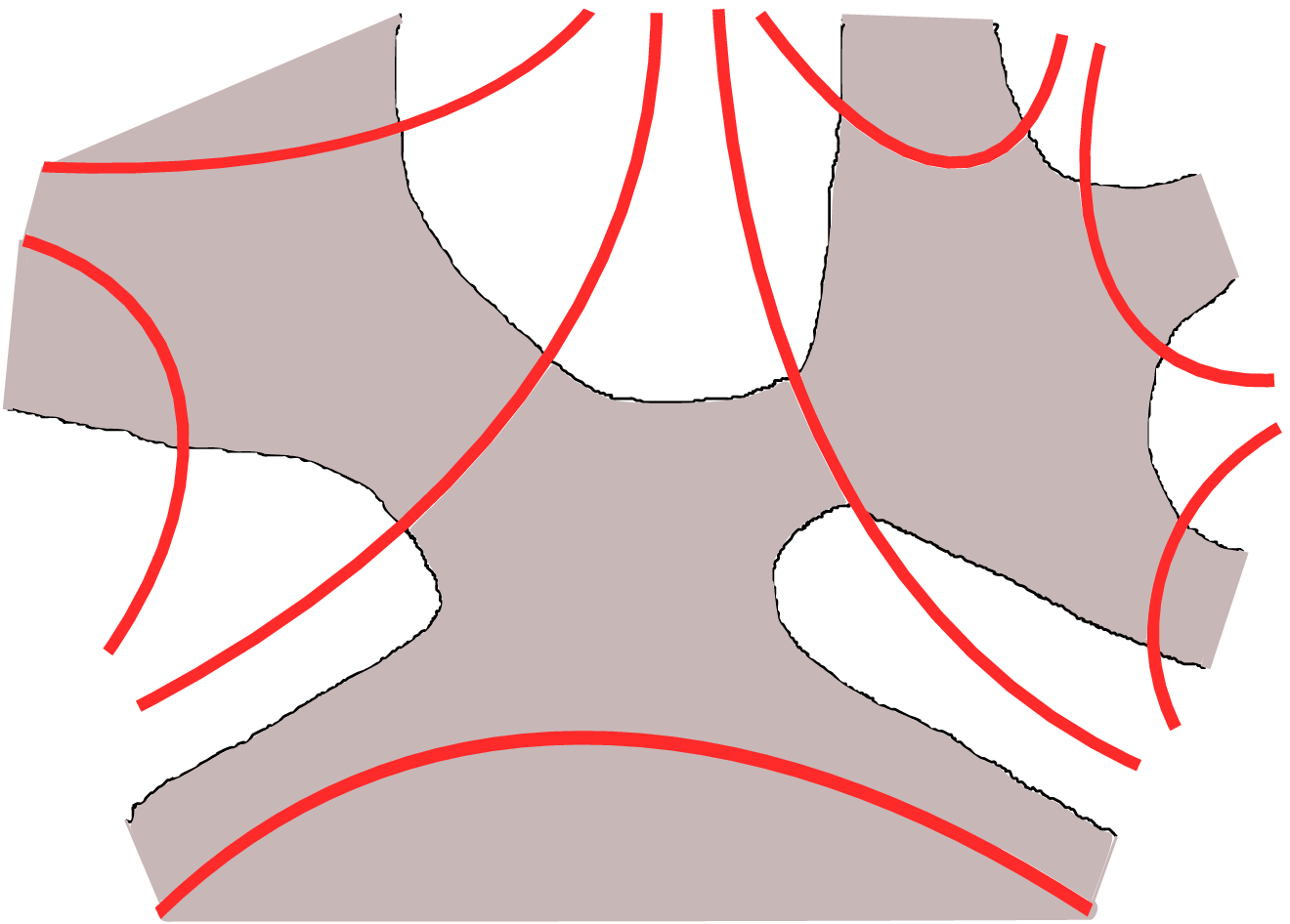}
		\put(14,45){\huge$g_{i+1} Z_{i+1}$}
		\put(42,25){\huge$g_i Z_i$}
		\put(63,42){\huge$g'_{i+1} Z_{i+1}$}	
	\end{overpic}
\end{figure}

The conclusion of this discussion is that the collection of all possible $g_i Z_i$ are arranged in a tree structure, with different $g_i Z_i$ intersecting only along portals as in (\ref{piecesintersect}), and that each map $\mu: g_i Z_i\to Y$ only depends on the space $g_i Z_i$. $\tilde{Y}$ is a tree of simply connected spaces, so is itself simply connected; and for $g_iP_i\subset g_i Z_i\cap g_{i+1} Z_{i+1}$ as above, we have a translate of $Z_{i+1}$ glued to $g_i Z_i$ along every coset $g_i h H^i_{P_i}$ for $h\in H^i$, so $\mu:\tilde{Y}\to Y$ really is a universal covering map.
\end{proof}

 We now promote $\tilde{Y}$ to a triple in $\mathcal{V}_{j-1}$ by adding the data of a group action and colourings. We  define these now, and afterwards we will verify the numbered properties of $\mathcal{V}_{j-1}$ given in Section 7.
 
 \begin{defn}(Group action and colourings for $\tilde{Y}$)\\
 
 Define the group
 \begin{equation*}
H(Y):=\{ g_i\,|\,g_i\text{ is an element of the form (\ref{gi}) with }Z_i=Z_0\in\mathcal{V}_j\}<G.
 \end{equation*}
 If $g_k$ is another element of the form (\ref{gi}), then $g_i g_k$ is also of the form (\ref{gi}), and so $g_i:g_k Z_k\to g_i g_k Z_k\subset\tilde{Y}$, hence we get $g_i\tilde{Y}\subset \tilde{Y}$. It is also clear that $\mu g_i=\mu$, so $H(Y)$ is a subgroup of the deck transformation group of $\mu:\tilde{Y}\to Y$. What's more it follows from the construction of $\mu$ and $\tilde{Y}$ that $H(Y)\cdot x=\mu^{-1}(\mu(x))$ for all $x\in\tilde{Y}$, hence $H(Y)$ is the full group of deck transformations, and acts freely cocompactly on $\tilde{Y}$.
 
To complete the triple we need to provide colourings $(c^Y_x)$ for every vertex in $\tilde{Y}$: if $x\in Z$ is a vertex and $(Z,H,(c_x))\in\mathcal{V}_j$ then endow each translate $g_ix$ with the colouring $c^Y_{g_ix}:=g_ic_x$. This defines a triple $(\tilde{Y}, H(Y),(c^Y_x))\in\mathcal{V}_{j-1}$. Doing this for all connected components $Y$ of $\mathcal{Y}$ defines the entire collection $\mathcal{V}_{j-1}$.
\end{defn}

We have already seen that $H(Y)$ acts freely cocompactly on $\tilde{Y}$, and it is clear from the definition that the colourings $(c^Y_x)$ are invariant under the action of $H(Y)$ - so in particular are well-defined. There is just one other thing we need to check before going on to the numbered properties of $\mathcal{V}_{j-1}$, and this is the following lemma.
\begin{lem}\label{inthalfspace}
$\tilde{Y}$ is an intersection of half-spaces in $X$.
\end{lem}
\begin{proof}
 For each $g_i Z_i$ in $\tilde{Y}$ as in Definition \ref{tildeTX}, and boundary wall $W$ of $Z_i$ of colour $<j$, consider the wall $g_iW$. We claim that $\tilde{Y}$ is an intersection of half-spaces corresponding to these walls. Any edge leaving $\tilde{Y}$ must cross one of these walls, so it suffices to show that each of these walls has a half-space containing $\tilde{Y}$. Indeed consider $g_i Z_i$ and $g_i W$ as above, and consider $P_i$ and $g_{i+1} Z_{i+1}$ as from Definition \ref{tildeTX}, so that we have $g_i Z_i$ glued to $g_{i+1} Z_{i+1}$ along $g_i P_i$. Let $P_i$ lie in a wall $W_i$ and let $M$ be the part of $\tilde{Y}$ on the opposite side of $g_i W_i$ to $g_i Z_i$. We will show that $M$ and $g_i Z_i$ lie entirely on the same side of $g_i W$.
 
 If $g_i W\cap g_i W_i=\emptyset$ then we are done. Now suppose $g_i W$ and $g_i W_i$ intersect, then apply Lemma \ref{intersectbwall} to find a vertex $x\in Z_i$ incident at edges $e, e_i$ which are dual to $W$ and $W_i$ respectively, and form the corner of a square in $X$. Let $g_{P_i}y$ be the vertex at the other end of $e_i$ (so $y\in Z_{i+1}$), then $g_{P_i}y$ is incident at an edge dual to $W$. $P_i$ is a $g_{P_i}$-teleport of some portal $P'_i\subset Z_{i+1}$, so
 \begin{equation*}
 [e_i,c^i_{e_i}]=[e_i,g_{P_i} c^{i+1}_{e_{i+1}}]
 \end{equation*}
 where $e_{i+1}:=g_{P_i}^{-1} e_i$ and $c^i, c^{i+1}$ are the colourings for $Z_i, Z_{i+1}$. So
 \begin{equation*}
 j>c_{e_i}(\overline{W})=g_{P_i} c^{i+1}_{e_{i+1}}(\overline{W})=g_{P_i} c^{i+1}_{y}(\overline{W}).
 \end{equation*}

 Thus $g_{P_i}^{-1} W$ is a boundary wall of $Z_{i+1}$ of colour $<j$, so $Z_{i+1}$ and $y$ lie entirely on the same side of it. In turn $g_{P_i} Z_{i+1}$ and $Z_i$ lie entirely on the same side of $W$, and $g_{i+1} Z_{i+1}$ and $g_i Z_i$ lie entirely on the same side of $g_i W$. Iterating this argument along branches of $M$ shows that $M$ and $g_i Z_i$ lie entirely on the same side of $g_i W$ as required.
 \begin{figure}[H]
 	\centering
 	\definecolor{RED}{HTML}{FF2A2A}
 	\definecolor{BLUE}{HTML}{0000FF}
 	\definecolor{GREEN}{HTML}{aad400}
 	\begin{overpic}[width=.55\textwidth,clip=true]{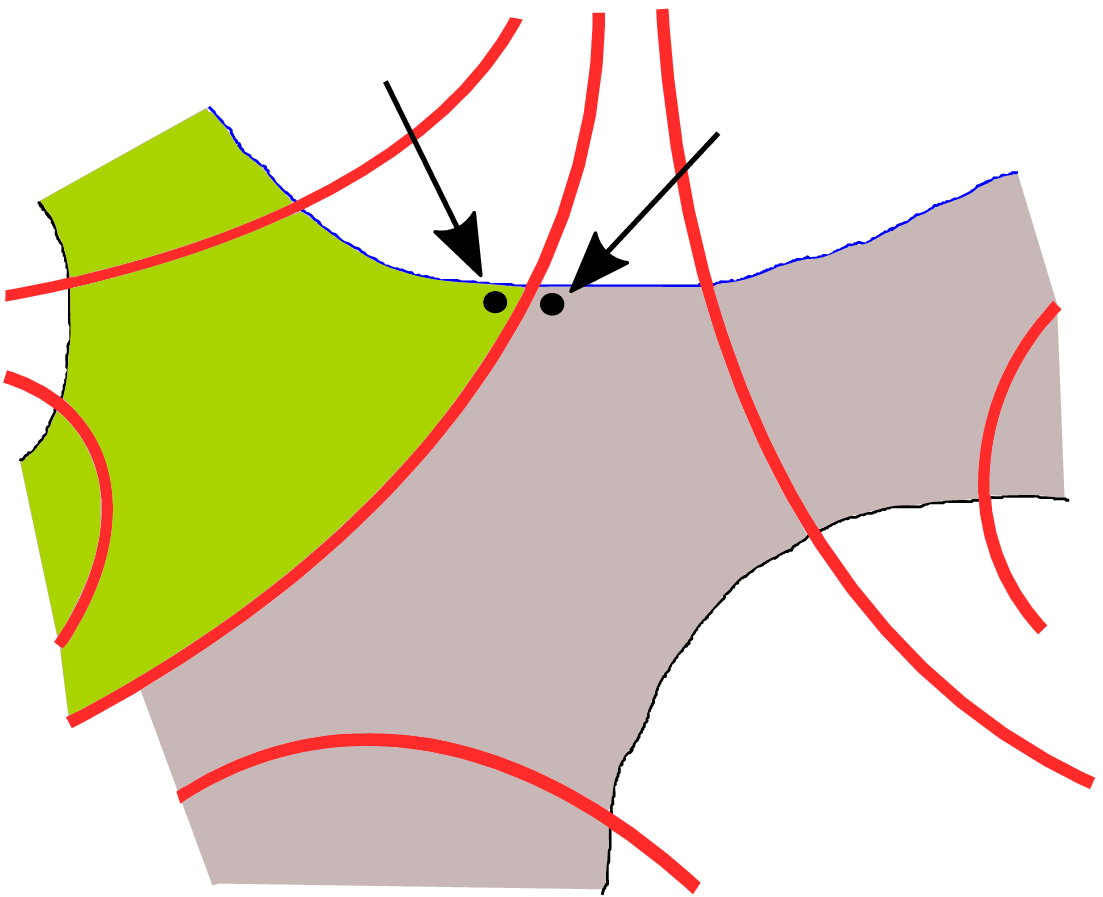}
 		\put(14,43){\huge$g_{i+1} Z_{i+1}$}
 		\put(46,30){\huge$g_i Z_i$}
 		\put(3,70){\huge\textcolor{GREEN}{$M$}}
 		\put(25,23){\huge\textcolor{RED}{$g_i W_i$}}
 		\put(25,77){\huge$g_ig_{P_i}y$}
 		\put(68,70){\huge$g_i x$}
 		\put(75,54){\huge\textcolor{BLUE}{$g_i W$}}
 	\end{overpic}
 \end{figure}
 
 \end{proof}

 Lastly we check the numbered properties of $\mathcal{V}_{j-1}$ from Section \ref{sec:startgluing}:
 \begin{enumerate}[(1)]
 \item If $e\in E(X)$ joins vertices $x,y\in\tilde{Y}$, we need $[e, c^Y_x]=[e,c^Y_y]$ to hold. This is clearly true if $e$ is contained within a translate $g_iZ_i$ because $c^Y$ is a translate of the colourings for $Z_i$. From the proof of Lemma \ref{muwelldefined}, we know that the different $g_iZ_i$ translates are separated by walls in a tree structure, and that two translates are adjacent only if they are of the form $g_iZ_i$ and $g_{i+1} Z_{i+1}$ as in Definition \ref{tildeTX}. We may assume that $h_{i+1}=1$, as changing $h_{i+1}$ doesn't change the translate $g_{i+1} Z_{i+1}$ - so $g_{i+1}=g_i g_{P_i}$. If $e$ crosses from $g_iZ_i$ to $g_{i+1} Z_{i+1}$, then it crosses the portal translate $g_iP_i=g_{i+1}P'_i$, and can be written $e=g_ie_i=g_{i+1}e'_i$ for $e_i$ dual to $P_i$ and $e'_i$ dual to $P'_i$. Suppose $c^i,c^{i+1}$ are the colourings for $Z_i,Z_{i+1}$, and that $x$ is the end of $e$ in $g_i Z_i$. Then putting $x=g_ix_i$ and $y=g_{i+1}y'_i$ we have
 \begin{align*}
 [e,c^Y_x]&=[e,g_ic^i_{x_i}]\\
 &=g_i[e_i,c^i_{x_i}]\\
&=g_i[e_i,c^i_{e_i}]\\
&=g_i g_{P_i}[e'_i,c^{i+1}_{e'_i}]&\text{since $P_i$ is a $g_{P_i}$-teleport of $P'_i$}\\
&=g_{i+1}[e'_i,c^{i+1}_{y'_i}]\\
&=[e,g_{i+1}c^{i+1}_{y'_i}]\\
&=[e,c^Y_y].
 \end{align*}
 \item Given $e\in E(X)$ joining vertices $x,y$ with $x\in\tilde{Y}$, $c^Y_x(e)>j$ if and only if $e$ is contained within some translate $g_iZ_i$; $c^Y_x(e)=j$ if and only if $e$ crosses a portal translate into a different $g_iZ_i$; so $y\in\tilde{Y}$ if and only if $c^Y_x(e)>j-1$ as required.
 \item Let $(H\cdot e, Z)\in\mathcal{V}_j^\pm(f,c)$, with $(Z,H,(c_x))\in\mathcal{V}_j$ and $Z/H$ in a component $Y$ of $\mathcal{Y}$. We have $Z=Z_i$ and $g_i Z_i\subset\tilde{Y}$ for some $g_i$ as in Definition \ref{tildeTX}, so we get $(H(Y)\cdot g_i e,\tilde{Y})\in\mathcal{V}_{j-1}^\pm(f,c)$. Moreover, $\mu:\tilde{Y}\to Y$ is a universal covering with deck transformation group $H(Y)$, so all choices of $g_i$ put $g_i e$ in the same $H(Y)$-orbit. This means that we have a bijection $\mathcal{V}_j^\pm(f,c)\to \mathcal{V}_{j-1}^\pm(f,c)$ for all $(f,c)$, and so the Gluing Equations hold in $\mathcal{V}_{j-1}$.
 \item Finally, $H(Y)\cong\pi_1(Y)\in\mathcal{QVH}$ by Lemma \ref{YQVH}.
 \end{enumerate}

\end{document}